\documentclass[11pt,a4paper]{article}
\usepackage{hyperref}
\usepackage[english]{babel}
\usepackage[utf8]{inputenc}
\usepackage[T1]{fontenc}
\usepackage{amsmath,amsthm,amssymb,amsfonts}
\numberwithin{equation}{section}
\usepackage{enumerate}
\usepackage{faktor}
\usepackage{dsfont}
\usepackage{scalerel,stackengine}
\usepackage{textcomp}
\usepackage{graphicx}
\usepackage{extarrows}
\usepackage{epigraph}
\usepackage{graphicx}
\usepackage{subcaption}
\usepackage{sidecap}
\usepackage{indentfirst}
\usepackage{tikz}
\usepackage{tikz-cd}
\usetikzlibrary{shapes.misc,arrows,matrix,patterns,decorations.markings,positioning}
\tikzset{cross/.style={cross out, draw=black, minimum size=2*(#1-\pgflinewidth), inner sep=0pt, outer sep=0pt},
cross/.default={4.5pt}}
\usepackage{pgfplots}
\usepackage{caption}
\usepackage{float}
\usepackage{color}
\usepackage{epstopdf}
\usepackage{epsfig}
\usepackage{stmaryrd}
\usepackage{color}
\usepackage{geometry}
\usepackage{multicol}
\usepackage{verbatim}

\DeclareMathOperator{\Ker}{Ker }

\renewcommand{\geq}{\geqslant}
\renewcommand{\leq}{\leqslant} 
\renewcommand{\epsilon}{\varepsilon}

\newcommand{\Z}{\mathbb{Z}}

\newcommand{\F}{\mathbb{F}}

\newcommand{\C}{\mathbb C}

\newcommand{\norm}[1]{\left\Vert#1\right\Vert}

\newcommand{\xist}{\xi_{\text{st}}}

\DeclareFontFamily{U}{mathx}{\hyphenchar\font45}
\DeclareFontShape{U}{mathx}{m}{n}{
      <5> <6> <7> <8> <9> <10>
      <10.95> <12> <14.4> <17.28> <20.74> <24.88>
      mathx10
      }{}
\DeclareSymbolFont{mathx}{U}{mathx}{m}{n}
\DeclareFontSubstitution{U}{mathx}{m}{n}
\DeclareMathAccent{\widecheck}{0}{mathx}{"71}
\DeclareMathAccent{\wideparen}{0}{mathx}{"75}

\newtheorem{teo}{Theorem}[section]
\newtheorem*{teo*}{Theorem}
\newtheorem{lemma}[teo]{Lemma}
\newtheorem{prop}[teo]{Proposition}
\newtheorem*{prop*}{Proposition}

\newtheorem{cor}[teo]{Corollary}

\theoremstyle{definition}
\newtheorem{remark}[teo]{Remark}

\usepackage{xpatch}
\makeatletter
\xpatchcmd{\@thm}{\thm@headpunct{.}}{\thm@headpunct{}}{}{}
\makeatother

\pgfplotsset{compat=1.13}
\begin{document}
\title{Detecting fibered strongly quasi-positive links}
\author{\scshape{Alberto Cavallo}\\ \\
 \footnotesize{Max Planck Institute for Mathematics,}\\
 \footnotesize{Bonn 53111, Germany}\\ \\ \small{cavallo@mpim-bonn.mpg.de}}
\date{}

\maketitle

\begin{abstract}
 We prove that an $n$-component fibered link $L$ in $S^3$ is strongly quasi-positive if and only if $\tau(L)=g_3(L)+n-1$, where $g_3(L)$ denotes the Seifert genus and $\tau$ is the Ozsv\'ath-Szab\'o concordance invariant. We also provide a table which contains a list of some fibered prime links with at most 9 crossings; and we explicitly determine the ones that are strongly quasi-positive and their maximal self-linking number.
\end{abstract}

\section{Introduction}
It is a very known result in low dimensional topology that every contact 3-manifold can be presented by an open book decomposition; this is a pair $(L,\pi)$ where $L$ is a smooth link in $M$ and $\pi:M\setminus L\rightarrow S^1$ is a locally trivial fibration whose fiber's closure is a compact surface $F$ with $\partial F=L$. 
Such an $L$ is then called a fibered link in $M$ and $F$ is a fibered surface for $L$.

In this paper we are only interested in the case where the ambient manifold $M$ is the 3-sphere. Fibered links in $S^3$ possess special properties; in fact, those links have unique fibered surfaces up to isotopy and, once such a surface is fixed, the fibration $\pi$ is also completely determined. It follows that in $S^3$ we can define fibered links as the links whose complement fibers over the circle. 

Since an open open book decomposition $(L,\pi)$ always determines a contact structure, we obtain a partition of fibered links in $S^3$ by saying that $L$ carries the structure $\xi_L$. In particular, it is a consequence of classical results of Harer and Stallings \cite{Harer,Stallings}, and the Giroux correspondence \cite{Giroux}, that the subset of fibered links, inducing the unique tight structure $\xist$ on $S^3$, coincides with the set of strongly quasi-positive links. We recall that a link is said strongly quasi-positive if it can be written as closure of the composition of $d$-braids of the form 
\[(\sigma_i\cdot\cdot\cdot\sigma_{j-2})\sigma_{j-1}(\sigma_i\cdot\cdot\cdot\sigma_{j-2})^{-1}\quad\text{ for some }\quad d\geq j\geq i+2\geq3\] or \[\sigma_i\quad\text{ for }\quad i=1,...,d-1\:,\]
where $\sigma_1,...,\sigma_{d-1}$ are the Artin generators of the $d$-braids group.
\begin{remark}
 Fibered and strongly quasi-positive links in $S^3$ are connected transverse $\C$-links in the sense of \cite{Cavallo3}.
\end{remark}
The main result of our paper is a criterion for detecting exactly which fibered links in $S^3$ are strongly quasi-positive, and then carry the contact structure $\xist$, generalizing results of Hedden \cite{Hedden} (for knots) and Boileau, Boyer and Gordon \cite{BBG} (for non-split alternating links).
We recall that Ozsv\'ath, Stipsicz and Szab\'o defined the $\tau$-set (of $2^{n-1}$ integers) of an $n$-component link in \cite{book}, using the link Floer homology group $cHFL^-$, and they denote with $\tau_{\max}$ and $\tau_{min}$ its maximum and minimum. Moreover, the $\tau$-set contains $\tau$, the concordance invariant defined in \cite{Cavallo} by the author.
\begin{teo}
 \label{teo:main}
 A fibered link $L$ in $S^3$ is strongly quasi-positive if and only if $\tau(L)=g_3(L)+n-1$ and, in this case, one also has $\tau(L)=\tau_{\max}(L)$.
\end{teo}
Combining Theorem \ref{teo:main} with \cite[Theorem 1.1]{EvHM} and \cite[Theorem 1.4]{Cavallo3} we easily obtain the following corollary.
\begin{cor}
 \label{cor:main}
 For a fibered link $L$ the fact that $\tau(L)=g_3(L)+n-1$ immediately implies $\text{SL}(L)=2\tau(L)-n$, where $\text{SL}(L)$ denotes the maximal self-linking number of $L$.
\end{cor}
Furthermore, we give a table where we list all prime alternating fibered links with crossing number at most 9, together with the prime non-alternating fibered ones with crossing number at most 7, and we denote whether they are strongly quasi-positive. We plan to expand the table in the future, but unfortunately at the moment we do not have knowledge of any program to compute $\tau$ for links.

The paper is organized as follows: in Section 2 we recall the basics properties of link Floer homology and fibered surfaces in the 3-sphere; moreover, we prove Theorem \ref{teo:main}. Finally, in Section 3 we present our table of small prime fibered links and we explain how to read it. 

\paragraph*{Funding:}
The author has a post-doctoral fellowship at the Max Planck Institute for Mathematics in Bonn.

\section{Proof of the main theorem}
\subsection{Link Floer homology}
In \cite{OSlinks} Ozsv\'ath and Szab\'o describe how to construct the chain complex $\left(\widehat{CFL}_*(\mathcal D),\widehat\partial\right)$ from a Heegaard diagram $\mathcal D=(\Sigma,\alpha,\beta,\textbf w,\textbf z)$, where the sets of basepoints both contain $n$ elements. They show that, under some conditions, the diagram $\mathcal D$ represents an $n$-component link $L$.

If we ignore the information given by $\textbf z$ then $\mathcal D$ is just a (multi-pointed) Heegaard diagram for $S^3$ and, using the same procedure, we obtain the complex $\left(\widehat{CF}_*(\mathcal D),\widehat\partial\right)$, see \cite{OS}. When $n=1$ the homology of the latter complex is denoted by $\widehat{HF}_*(S^3)\cong\F_0$, where $\F$ is the field with two elements. In general, we have that 
\[H_*\left(\widehat{CF}(\mathcal D)\right)\cong\widehat{HF}_*(S^3)\otimes\left(\F_{-1}\oplus\F_0\right)^{\otimes\:n-1}\:;\] hence, if $\mathcal D_1$ and $\mathcal D_2$ have the same number of basepoints then $\widehat{CF}_*(\mathcal D_1)$ is chain homotopy equivalent to $\widehat{CF}_*(\mathcal D_2)$.

The homology of $\widehat{CFL}_*(\mathcal D)$, following the notation of the author in \cite{Cavallo}, is denoted by $\widehat{\mathcal{HFL}}_*(L)$, but for what we said before this group is graded isomorphic to $H_*\left(\widehat{CF}(\mathcal D)\right)$. The difference, in the case of links, is given by the fact that the basepoints in $\textbf z$ define the Alexander (collapsed) increasing filtration $\mathcal A^s\widehat{CFL}_*(\mathcal D)$ with $s\in\Z$.
Such a filtration descends into homology in the following way: consider the quotient projection $\pi_d:\Ker\widehat\partial_d\rightarrow\widehat{\mathcal{HFL}}_d(L)$, where $\widehat\partial_d:\widehat{CFL}_d(\mathcal D)\rightarrow\widehat{CFL}_{d-1}(\mathcal D)$ is the restriction of $\widehat\partial$ to $\widehat{CFL}_d(\mathcal D)$, and define \[\Ker\widehat\partial_{d,s}=\Ker\widehat\partial_d\cap\mathcal A^s\widehat{CFL}_d(\mathcal D)\:.\] 
Then we say that \[\mathcal A^s\widehat{\mathcal{HFL}}_d(L)=\pi_d(\Ker\widehat\partial_{d,s})\subset\widehat{\mathcal{HFL}}_d(L)\] for any $d\in\Z$. It is important to observe that $\mathcal A^s\widehat{\mathcal{HFL}}_*(L)$ is not the same homology group obtained by just taking the homology of the Alexander level $s$ of $\widehat{CFL}(\mathcal D)$, which is instead denoted with $H_*\left(\mathcal A^s\widehat{CFL}(\mathcal D)\right)$; more specifically, the latter group is obtained by first restricting $\widehat{CFL}_d(\mathcal D)$ and $\widehat\partial_d$ to the Alexander level $s$ and then extracting the homology. In particular, note that $\mathcal A^s\widehat{\mathcal{HFL}}_*(L)$ is, by definition, always a subgroup of $\widehat{\mathcal{HFL}}_*(L)$, while this needs not be true for $H_*\left(\mathcal A^s\widehat{CFL}(\mathcal D)\right)$.

In addition, we write $\widehat{HFL}_{*,*}(L)$ for the homology of the graded object associated to $\mathcal A\widehat{CFL}_*(\mathcal D)$, which is the bigraded complex $(\text{gr}_{*,*}(\mathcal D),\text{gr}(\widehat\partial))$ defined as follows, see \cite{book}: we have that $\text{gr}_{d,s}(\mathcal D)$ is the subspace of $\mathcal A^s\widehat{CFL}_d(\mathcal D)$ spanned by generators that are not in $\mathcal A^{s-1}\widehat{CFL}_d(\mathcal D)$, while $\text{gr}(\widehat\partial)$ is the component of $\widehat\partial$ which preserves the Alexander grading, so that
$\text{gr}_{d,s}(\widehat\partial):=\text{gr}(\widehat\partial)\lvert_{\text{gr}_{d,s}(\mathcal D)}:\text{gr}_{d,s}(\mathcal D)\rightarrow\text{gr}_{d-1,s}(\mathcal D)$.

We conclude this subsection by recalling that in \cite[Theorem 1.3]{Cavallo} the invariant $\tau_{\min}(L)$ was shown by the author to coincide with the minimal integer $s$ such that $\mathcal A^s\widehat{\mathcal{HFL}}_*(L)$ is non-zero.

\subsection{Fibered surfaces and the Thurston norm}
It is a result of Ghiggini \cite{Ghiggini} and Ni \cite{Ni} that link Floer homology detects fibered links in the 3-sphere.
\begin{teo}[Ghiggini-Ni]
 \label{teo:Ghiggini-Ni}
 A link $L\hookrightarrow S^3$ with $n$-components is fibered if and only if $\dim\widehat{HFL}_{*,s_{\text{top}}}(L)=1$, where $s_{\text{top}}$ is the maximal $s\in\Z$ such that the group $\widehat{HFL}_{*,s}(L)$ is non-zero.
\end{teo}
We recall that $\norm{L}_T$ denotes the evaluation of the Thurston semi-norm \cite{Thurston} at the homology class represented by the Seifert surfaces of $L$; and it is defined as
\[\norm{L}_T=o(L)-\max\left\{\chi(\Sigma)\right\}\:,\] where $\Sigma$ is a compact and oriented surface in $S^3$, such that $\partial\Sigma=L$, and $o(L)$ is the number of unknotted unknots in $L$.

Hence, it follows from \cite{Ninorm} that, for an $L$ as in Theorem \ref{teo:Ghiggini-Ni}, one has $2s_{\text{top}}=\norm{L}_T+n$, assuming that $L$ is not the unknot. Morever, we have the following result.
\begin{prop}
 \label{prop:fibered}
 If $L$ is a fibered $n$-component link in $S^3$ then $s_{\text{top}}=g_3(L)+n-1$. 
\end{prop}
\begin{proof}
 The statement is clearly true for the unknot and it is a classical result, see \cite[Chapters 4 and 5]{Kawauchi}, that fibered surfaces are always connected and minimize the Thurston norm of $L$. Therefore, we have \[s_{\text{top}}=\dfrac{-\chi(F)+n}{2}=\dfrac{-(2-2g(F)-n)+n}{2}=g(F)+n-1\:,\] where $F$ is a fibered surface for $L$. 
 
 Now, if $L$ admits a Seifert surface $T$, with $g(T)<g(F)$, then this contradicts the fact that $F$ minimizes $\norm{L}_T$ and the claim follows. 
\end{proof}

\subsection{The contact invariant \texorpdfstring{$c(L)$}{c(L)}}
In order to prove Theorem \ref{teo:main}, we need to recall the original definition of the contact invariant $\widehat c(\xi)\in\widehat{HF}(S^3)$ from \cite{OSz} and adapt it to our setting.
In \cite{OSz} Ozsv\'ath and Szab\'o, for a given fibered knot $K\hookrightarrow S^3$, constructed a Heegaard diagram $\mathcal D_K$, representing the mirror image of the knot $K$, with a distinguished intersection point $c(K)$ such that \[A(c(K))=-g_3(K)\quad\text{ and }\quad A(y)>-g_3(K)\] for every other intersection point $y$ in $\mathcal D_K$. In particular, this gives that \[H_*\left(\mathcal A^{-g_3(K)}\widehat{CFK}(\mathcal D_{K})\right)\cong\F\] and \[H_*\left(\mathcal A^{s}\widehat{CFK}(\mathcal D_{K})\right)\cong\{0\}\quad\text{ for any }s<-g_3(K)\:.\]
This construction works exactly in the same way for an $n$-component fibered link $L$: only that now one has $A(c(L))=-g$ with $g\in\Z$ and $c(L)$ is still the unique minimal Alexander grading intersection point in $\mathcal D_L$.
\begin{lemma}
 \label{lemma:g}
 Suppose that $L,c(L)$ and $g$ are as before. Then we have that $g=g_3(L)+n-1$.
\end{lemma}
\begin{proof}
 The element $c(L)$ is necessarily a cycle and $[c(L)]$ has to be non-zero in $\widehat{HFL}(L^*)$; hence, we have that $A(c(L))$ is the minimal integer $s$ such that $\widehat{HFL}_{*,s}(L^*)$ is non-zero.
 Because of the symmetries of $\widehat{HFL}(L)$, see \cite{OSz2}, and Proposition \ref{prop:fibered} this means that $A(c(L))=-g_3(L)-n+1=-g$.
\end{proof}
If we consider the inclusion $i^{L,-g}:\mathcal A^{-g}\widehat{CFL}(\mathcal D_L)\hookrightarrow\widehat{CF}(\mathcal D_L)$ then we can see $c(L)$ as a cycle in $\widehat{CF}(\mathcal D_L)$ by dropping the information about the Alexander grading.

The element $c(L)$ can then be identified with the contact element defined in \cite{Cavallo2}, using multi-pointed Legendrian Heegaard diagrams, generalizing the presentation of Honda, Kazez and Mati\'c, see \cite{HKM}. This means that 
\begin{equation}
\label{c}
[c(L)]=\widehat c(\xi_L)\otimes\textbf e_{1-n}\in \widehat{HF}(S^3)\otimes\left(\F_{-1}\oplus\F_0\right)^{\otimes\:n-1}\:,
\end{equation}
where $\textbf e_{1-n}$ is the unique generator of $(\F_{-1}\oplus\F_0)^{\otimes\:n-1}$ in grading $1-n$. In particular, we can state the following lemma.
\begin{lemma}
 \label{lemma:vanishing}
 Suppose that $L$ is a fibered link and $\xi_L$ is the contact structure on $S^3$ carried by $L$. Then $c(L)$ is zero in homology whenever $\xi_L$ is overtwisted.
\end{lemma}
\begin{proof}
 From \cite{OSz} one has $\widehat c(\xi_L)=[0]$ when $\xi_L$ is overtwisted.
 Hence, the homology class $[c(L)]$ is equal to zero for Equation \eqref{c}.
\end{proof}
From the definition of $c(L)$ we observe that $c(L)$ is non-zero in homology if and only if $-g$ is the minimal $t$ such that the map 
\begin{equation}
    i_*^{L,t}:H_*\left(\mathcal A^t\widehat{CFL}(\mathcal D_L)\right)\longrightarrow\widehat{\mathcal{HFL}}_*(L^*)\cong\widehat{HF}(S^3)\otimes\left(\F_{-1}\oplus\F_0\right)^{\otimes\:n-1}
    \label{A}
\end{equation}
induced by the inclusion $i^{L,t}$ is non-trivial.
\begin{proof}[Proof of Theorem \ref{teo:main}]
 If $L$ is strongly quasi-positive then from \cite[Theorem 1.4]{Cavallo3} one has $\tau(L)=\tau_{\max}(L)=g_3(L)+n-1$. Suppose now that $\tau_{\max}(L)=g_3(L)+n-1=g$, see Lemma \ref{lemma:g}; we want to prove that $L$ is strongly quasi-positive.
 
 We have that 
 \begin{equation}
     -g=-\tau_{\max}(L)=\tau_{\min}(L^*)=\min_{s\in\Z}\left\{\mathcal A^s\widehat{\mathcal{HFL}}(L^*)\text{ is non-zero}\right\}=\min_{s\in\Z}\{i_*^{L,s}\text{ is non-trivial}\}
  \label{B}
 \end{equation}     
 from the definition of $\tau$-set given in \cite{Cavallo}. Since $L$ is fibered, the combination of Equations \eqref{A} and \eqref{B} tells us precisely that $c(L)$ is non-zero in homology and, applying Lemma \ref{lemma:vanishing}, that $\xi_L=\xist$. 
 
 At this point, we conclude by recalling that a fibered link $L$ carries the tight structure on $S^3$ if and only if is strongly quasi-positive, as we observed in the introduction.
\end{proof}

\section{Table of small prime fibered strongly quasi-positive links}
Using Theorem \ref{teo:main} and some computation\footnote{The computations have been obtained with the help of Open AI ChatGPT 6 Astra.} we can detect exactly which prime links are fibered and, among these, which ones are strongly quasi-positive. Those data are enclosed in Table \ref{tab:small_links}.
We denote the links following the notation on Linkinfo database \cite{Linkinfo}. Therefore, the number inside brackets identifies the relative orientation of the link.
Clearly, being fibered and strongly quasi-positive does depend on relative orientations. Moreover, a link $L$ is fibered if and only if $L^*$ is fibered; while when in Table \ref{tab:small_links} we write that $L$ is strongly quasi-positive, we mean that at least one between $L$ and $L^*$ is. In fact, the following result holds.
\begin{teo}
 \label{teo:main2}
 Suppose that $L$ is a strongly quasi-positive link in $S^3$. Then its mirror image $L^*$ is also strongly quasi-positive if and only if $L$ is an unlink. 
\end{teo}
Theorem \ref{teo:main2} follows from a more general result of Hayden \cite{Hayden}, but we can prove it directly using only link Floer homology.
Let us start from the following lemma.
\begin{lemma}
 \label{lemma:unknots}
 If $L$ is strongly quasi-positive then the same its true for $L'$, where $L=L'\sqcup\bigcirc_{o(L)}$ and $\bigcirc_{\ell}$ is the $\ell$-unlink. Furthermore, one has $\norm{L}_T=\norm{L'}_T$.
\end{lemma}
\begin{proof}
 The link $L$ bounds a quasi-positive surface $\Sigma$. It is known that quasi-positive surfaces always minimize the Thurston norm; hence, since one has $\norm{L}_T=\norm{L'}_T$ from the additivity of $\norm{\cdot}_T$, we have that $\Sigma$ is the disjoint union of $\Sigma'$ with some disks, where $\partial \Sigma'=L'$. This implies that $L'$ is strongly quasi-positive.
\end{proof}
At this point, in order to prove Theorem \ref{teo:main2} we only need Lemma \ref{lemma:tau}, which solves Exercise 8.4.9 (a) in \cite{book}. 
\begin{lemma}
 \label{lemma:tau}
 For every $n$-component link $L$ we have \[0\leq\tau_{\max}(L)-\tau_{\min}(L)\leq n-1\:.\] 
\end{lemma}
\begin{proof}
 It follows from \cite{CC} because $\tau_{\min}(L)=-\tau_{\max}(L^*)$, for the symmetries of $cHFL^-(L)$, and $\tau_{\max}(L)$ is a slice-torus concordance invariant. To show this we apply \cite[Lemma 2.1]{CC}; then the only non-trivial fact to prove is that $\tau_{\max}$ coincides with $\tau$ for torus links: this follows from \cite[Theorem 1.4]{Cavallo3}.
\end{proof}
In particular, this result implies that $\tau(L)+\tau(L^*)\leq n-1$.
\begin{proof}[Proof of Theorem \ref{teo:main2}]
 From Lemma \ref{lemma:unknots} we can suppose that $o(L)=0$. By using \cite[Theorem 1.4]{Cavallo3} again, the fact that $L$ and $L^*$ are both strongly quasi-positive gives $2\tau(L)-n=2\tau(L^*)-n$, since $\norm{L}_T=\norm{L^*}_T$. Therefore, from Lemma \ref{lemma:tau} one has
 \[2\tau(L)\leq n-1\quad\text{ and then }\quad\norm{L}_T\leq-1\:,\] but this is impossible.
\end{proof}
The following easier criterion can be used for alternating links. We recall that fibered links are always non-split and we denote by $\sigma(L)$ the signature and by $\nabla_L$ the Conway polynomial of $L$.
\begin{prop}
 A quasi-alternating $n$-component link is fibered and strongly quasi-positive if and only if $\nabla_L(z)=z^{-\sigma(L)}+f(z)$, where $f(z)$ has degree strictly smaller than $-\sigma(L)$.
\end{prop}
\begin{proof}
It follows from Theorem \ref{teo:Ghiggini-Ni} and the facts that $\widehat{HFL}(L)$ categorifies the polynomial $\left(t^{\frac{1}{2}}-t^{-\frac{1}{2}}\right)^{n-1}\cdot\nabla_L\left(t^{\frac{1}{2}}-t^{-\frac{1}{2}}\right)$ and quasi-alternating links are $\widehat{HFL}$-thin, see \cite{book}.
\end{proof}
This result agrees with \cite{BBG}, where is shown that, for a non-split alternating link $L$, being strongly quasi-positive is equivalent to being positive and this happens if and only if $2g_3(L)=1-n-\sigma(L)$. In fact, we also recall that for such a link $L$ one has $2\tau(L)=n-1-\sigma(L)$.

\clearpage
\begingroup
\makeatletter
\setlength{\@fptop}{0pt}
\makeatother
\captionsetup{font=small,skip=4pt}
\begin{table}[p]
\centering
\small
\setlength{\tabcolsep}{4pt}
\begin{tabular}[t]{cc}
name & strongly QP (SL)\\
\hline\hline
L2a1\{0\} & $\mathrm{Y}^{*}$ (0)\\
L2a1\{1\} & Y (0)\\
L4a1\{1\} & Y (2)\\
L5a1\{0\} & N\\
L5a1\{1\} & N\\
L6a1\{0\} & N\\
L6a3\{0\} & $\mathrm{Y}^{*}$ (4)\\
L6a4\{0,0\} & N\\
L6a4\{1,0\} & N\\
L6a4\{0,1\} & N\\
L6a4\{1,1\} & N\\
L6a5\{1,0\} & N\\
L6a5\{0,1\} & N\\
L6a5\{1,1\} & N\\
L6n1\{0,0\} & N\\
L6n1\{1,0\} & N\\
L6n1\{0,1\} & $\mathrm{Y}^{*}$ (3)\\
L6n1\{1,1\} & N\\
L7a1\{0\} & N\\
L7a1\{1\} & N\\
L7a2\{1\} & N\\
L7a3\{0\} & N\\
L7a3\{1\} & N\\
L7a5\{1\} & N\\
L7a6\{0\} & N\\
L7a7\{0,0\} & N\\
L7a7\{1,0\} & N\\
L7a7\{0,1\} & N\\
L7n1\{0\} & $\mathrm{Y}^{*}$ (4)\\
L7n1\{1\} & N\\
L7n2\{0\} & N\\
L7n2\{1\} & N\\
L8a1\{0\} & N\\
L8a1\{1\} & N\\
L8a2\{0\} & N\\
L8a2\{1\} & N\\
L8a3\{1\} & N\\
L8a4\{0\} & N\\
L8a4\{1\} & N\\
L8a5\{0\} & N\\
L8a7\{1\} & N\\
L8a8\{0\} & N\\
L8a8\{1\} & N\\
L8a9\{0\} & N\\
L8a9\{1\} & N\\
L8a10\{1\} & N\\
L8a14\{0\} & $\mathrm{Y}^{*}$ (6)\\
L8a15\{0,0\} & N\\
L8a16\{0,0\} & N\\
L8a16\{1,1\} & N\\
L8a17\{1,1\} & N\\
L8a18\{0,0\} & N\\
L8a18\{1,1\} & N\\
L8a19\{0,0\} & N\\
\end{tabular}%\hspace{.01\textwidth}
\begin{tabular}[t]{|cc}
name & strongly QP (SL)\\
\hline\hline
L8a19\{1,1\} & N\\
L8a20\{0,0\} & N\\
L8a20\{1,0\} & N\\
L8a21\{1,0,0\} & N\\
L8a21\{0,1,0\} & N\\
L8a21\{0,0,1\} & N\\
L8a21\{1,0,1\} & N\\
L8a21\{0,1,1\} & N\\
L8a21\{1,1,1\} & N\\
L8n1\{0\} & N\\
L8n1\{1\} & N\\
L8n2\{0\} & N\\
L8n2\{1\} & N\\
L8n3\{0,0\} & $\mathrm{Y}^{*}$ (5)\\
L8n3\{1,0\} & N\\
L8n3\{0,1\} & N\\
L8n3\{1,1\} & N\\
L8n4\{0,0\} & N\\
L8n4\{1,1\} & N\\
L8n5\{0,0\} & N\\
L8n5\{1,0\} & N\\
L8n5\{0,1\} & N\\
L8n5\{1,1\} & N\\
L8n6\{0,0\} & $\mathrm{Y}^{*}$ (3)\\
L8n6\{1,0\} & Y (5)\\
L8n6\{0,1\} & N\\
L8n6\{1,1\} & N\\
L8n7\{0,1,0\} & N\\
L8n7\{1,1,0\} & N\\
L8n7\{0,0,1\} & N\\
L8n7\{0,1,1\} & N\\
L8n8\{1,0,1\} & Y (4)\\
L8n8\{0,1,1\} & $\mathrm{Y}^{*}$ (4)\\
L9a2\{0\} & N\\
L9a2\{1\} & N\\
L9a5\{0\} & N\\
L9a6\{1\} & N\\
L9a8\{0\} & N\\
L9a8\{1\} & N\\
L9a9\{0\} & N\\
L9a9\{1\} & N\\
L9a11\{0\} & N\\
L9a12\{1\} & N\\
L9a14\{0\} & N\\
L9a14\{1\} & N\\
L9a16\{0\} & N\\
L9a20\{0\} & N\\
L9a21\{0\} & N\\
L9a22\{0\} & N\\
L9a24\{1\} & N\\
L9a26\{1\} & N\\
L9a27\{0\} & N\\
L9a28\{0\} & N\\
L9a29\{0\} & N\\
\end{tabular}%\hspace{.01\textwidth}
\begin{tabular}[t]{|cc}
name & strongly QP (SL)\\
\hline\hline
L9a31\{0\} & N\\
L9a32\{1\} & N\\
L9a33\{0\} & N\\
L9a36\{0\} & N\\
L9a38\{0\} & N\\
L9a39\{0\} & N\\
L9a41\{0\} & N\\
L9a42\{0\} & N\\
L9a42\{1\} & N\\
L9a43\{1,0\} & N\\
L9a43\{0,1\} & N\\
L9a43\{1,1\} & N\\
L9a44\{0,0\} & N\\
L9a44\{1,0\} & N\\
L9a44\{0,1\} & N\\
L9a46\{0,0\} & N\\
L9a46\{1,0\} & N\\
L9a46\{0,1\} & N\\
L9a46\{1,1\} & N\\
L9a47\{1,0\} & N\\
L9a47\{0,1\} & N\\
L9a47\{1,1\} & N\\
L9a48\{0,0\} & N\\
L9a49\{1,0\} & N\\
L9a49\{0,1\} & N\\
L9a50\{0,0\} & N\\
L9a50\{1,0\} & N\\
L9a50\{1,1\} & N\\
L9a51\{0,0\} & N\\
L9a51\{1,0\} & N\\
L9a51\{1,1\} & N\\
L9a52\{1,0\} & N\\
L9a52\{1,1\} & N\\
L9a53\{0,0\} & N\\
L9a53\{1,0\} & N\\
L9a53\{0,1\} & N\\
L9a53\{1,1\} & N\\
L9a54\{0,0\} & N\\
L9a54\{1,0\} & N\\
L9a54\{0,1\} & N\\
L9a54\{1,1\} & N\\
L9a55\{0,0,0\} & N\\
L9a55\{0,1,0\} & N\\
L9a55\{0,0,1\} & N\\
L9a55\{1,0,1\} & N\\
L9n3\{0\} & N\\
L9n3\{1\} & N\\
L9n4\{0\} & $\mathrm{Y}^{*}$ (6)\\
L9n4\{1\} & N\\
L9n5\{0\} & N\\
L9n5\{1\} & N\\
L9n6\{0\} & N\\
L9n6\{1\} & N\\
L9n8\{0\} & N\\
\end{tabular}
\end{table}
\begin{table}[p]
\centering
\small
\setlength{\tabcolsep}{4pt}
\begin{tabular}[t]{cc}
name & strongly QP (SL)\\
\hline\hline
L9n8\{1\} & N\\
L9n9\{0\} & $\mathrm{Y}^{*}$ (4)\\
L9n9\{1\} & N\\
L9n10\{0\} & N\\
L9n10\{1\} & N\\
L9n11\{0\} & N\\
L9n11\{1\} & N\\
L9n12\{0\} & N\\
L9n12\{1\} & Y (6)\\
L9n13\{0\} & N\\
L9n14\{0\} & N\\
L9n14\{1\} & N\\
L9n15\{0\} & $\mathrm{Y}^{*}$ (6)\\
L9n15\{1\} & N\\
L9n16\{0\} & N\\
L9n17\{0\} & N\\
L9n18\{0\} & $\mathrm{Y}^{*}$ (6)\\
L9n18\{1\} & Y (2)\\
L9n19\{0\} & $\mathrm{Y}^{*}$ (4)\\
L9n19\{1\} & Y (4)\\
L9n20\{0,1\} & N\\
L9n21\{0,1\} & $\mathrm{Y}^{*}$ (5)\\
L9n22\{1,0\} & N\\
L9n22\{0,1\} & N\\
L9n22\{1,1\} & N\\
L9n23\{0,0\} & N\\
L9n23\{1,0\} & N\\
L9n23\{0,1\} & N\\
L9n23\{1,1\} & N\\
L9n24\{0,0\} & N\\
L9n24\{1,0\} & N\\
L9n24\{0,1\} & N\\
L9n24\{1,1\} & N\\
L9n25\{0,0\} & N\\
L9n25\{1,0\} & N\\
L9n25\{0,1\} & N\\
L9n25\{1,1\} & N\\
L9n26\{0,0\} & N\\
L9n26\{1,0\} & N\\
L9n26\{0,1\} & N\\
L9n26\{1,1\} & N\\
L9n28\{0,0\} & N\\
L9n28\{1,1\} & N\\
L10a1\{0\} & N\\
L10a1\{1\} & N\\
L10a2\{0\} & N\\
L10a2\{1\} & N\\
L10a4\{0\} & N\\
L10a4\{1\} & N\\
L10a5\{0\} & N\\
L10a5\{1\} & N\\
L10a6\{0\} & N\\
L10a6\{1\} & N\\
L10a8\{0\} & N\\
\end{tabular}%\hspace{.01\textwidth}
\begin{tabular}[t]{|cc}
name & strongly QP (SL)\\
\hline\hline
L10a8\{1\} & N\\
L10a9\{0\} & N\\
L10a9\{1\} & N\\
L10a11\{1\} & N\\
L10a12\{0\} & N\\
L10a13\{1\} & N\\
L10a14\{0\} & N\\
L10a14\{1\} & N\\
L10a15\{1\} & N\\
L10a18\{0\} & N\\
L10a18\{1\} & N\\
L10a20\{0\} & N\\
L10a20\{1\} & N\\
L10a21\{0\} & N\\
L10a21\{1\} & N\\
L10a22\{0\} & N\\
L10a22\{1\} & N\\
L10a23\{0\} & N\\
L10a23\{1\} & N\\
L10a25\{1\} & N\\
L10a26\{1\} & N\\
L10a27\{0\} & N\\
L10a27\{1\} & N\\
L10a28\{0\} & N\\
L10a28\{1\} & N\\
L10a29\{0\} & N\\
L10a29\{1\} & N\\
L10a33\{1\} & N\\
L10a35\{0\} & N\\
L10a35\{1\} & N\\
L10a37\{0\} & N\\
L10a38\{1\} & N\\
L10a39\{0\} & N\\
L10a39\{1\} & N\\
L10a43\{0\} & N\\
L10a44\{0\} & N\\
L10a46\{0\} & N\\
L10a49\{1\} & N\\
L10a51\{0\} & N\\
L10a51\{1\} & N\\
L10a52\{0\} & N\\
L10a52\{1\} & N\\
L10a53\{0\} & N\\
L10a54\{0\} & N\\
L10a55\{0\} & N\\
L10a56\{0\} & N\\
L10a56\{1\} & N\\
L10a57\{1\} & N\\
L10a59\{1\} & N\\
L10a60\{0\} & N\\
L10a61\{1\} & N\\
L10a62\{0\} & N\\
L10a62\{1\} & N\\
L10a63\{0\} & N\\
\end{tabular}%\hspace{.01\textwidth}
\begin{tabular}[t]{|cc}
name & strongly QP (SL)\\
\hline\hline
L10a63\{1\} & N\\
L10a66\{0\} & N\\
L10a67\{0\} & N\\
L10a68\{0\} & N\\
L10a69\{0\} & N\\
L10a70\{0\} & N\\
L10a70\{1\} & N\\
L10a71\{0\} & N\\
L10a71\{1\} & N\\
L10a72\{1\} & N\\
L10a76\{0\} & N\\
L10a78\{0\} & N\\
L10a79\{1\} & N\\
L10a80\{1\} & N\\
L10a84\{0\} & N\\
L10a84\{1\} & N\\
L10a85\{0\} & N\\
L10a86\{0\} & N\\
L10a86\{1\} & N\\
L10a87\{1\} & N\\
L10a88\{1\} & N\\
L10a89\{0\} & N\\
L10a90\{0\} & N\\
L10a91\{1\} & N\\
L10a92\{0\} & N\\
L10a93\{0\} & N\\
L10a95\{1\} & N\\
L10a97\{0\} & N\\
L10a99\{1\} & N\\
L10a101\{1\} & N\\
L10a103\{0\} & N\\
L10a104\{0\} & N\\
L10a105\{0\} & N\\
L10a106\{0\} & N\\
L10a108\{1\} & N\\
L10a109\{0\} & N\\
L10a109\{1\} & N\\
L10a110\{0\} & N\\
L10a111\{0\} & N\\
L10a111\{1\} & N\\
L10a112\{0\} & N\\
L10a112\{1\} & N\\
L10a113\{0\} & N\\
L10a113\{1\} & N\\
L10a118\{0\} & $\mathrm{Y}^{*}$ (8)\\
L10a122\{0,0\} & N\\
L10a122\{1,0\} & N\\
L10a122\{0,1\} & N\\
L10a123\{1,0\} & N\\
L10a123\{0,1\} & N\\
L10a123\{1,1\} & N\\
L10a124\{0,0\} & N\\
L10a124\{1,0\} & N\\
L10a124\{0,1\} & N\\
\end{tabular}
\end{table}
\begin{table}[p]
\centering
\small
\setlength{\tabcolsep}{4pt}
\begin{tabular}[t]{cc}
name & strongly QP (SL)\\
\hline\hline
L10a125\{0,0\} & N\\
L10a127\{1,0\} & N\\
L10a127\{0,1\} & N\\
L10a128\{0,0\} & N\\
L10a128\{1,0\} & N\\
L10a128\{0,1\} & N\\
L10a129\{0,0\} & N\\
L10a129\{1,0\} & N\\
L10a129\{0,1\} & N\\
L10a130\{1,0\} & N\\
L10a130\{0,1\} & N\\
L10a131\{1,0\} & N\\
L10a132\{0,0\} & N\\
L10a134\{0,0\} & N\\
L10a134\{0,1\} & N\\
L10a134\{1,1\} & N\\
L10a135\{0,1\} & N\\
L10a136\{0,0\} & N\\
L10a136\{1,1\} & N\\
L10a138\{0,0\} & N\\
L10a138\{1,1\} & N\\
L10a139\{1,1\} & N\\
L10a140\{0,0\} & N\\
L10a140\{1,1\} & N\\
L10a142\{1,1\} & N\\
L10a145\{0,0\} & N\\
L10a145\{1,1\} & N\\
L10a146\{1,0\} & N\\
L10a148\{0,0\} & N\\
L10a148\{1,1\} & N\\
L10a149\{1,0\} & N\\
L10a149\{0,1\} & N\\
L10a150\{1,0\} & N\\
L10a150\{0,1\} & N\\
L10a151\{0,0\} & N\\
L10a151\{1,0\} & N\\
L10a151\{0,1\} & N\\
L10a151\{1,1\} & N\\
L10a152\{0,0\} & N\\
L10a153\{1,1\} & N\\
L10a154\{0,1\} & N\\
L10a154\{1,1\} & N\\
L10a155\{0,0\} & N\\
L10a156\{0,0\} & N\\
L10a156\{0,1\} & N\\
L10a157\{1,0\} & N\\
L10a158\{0,1\} & N\\
L10a160\{1,0\} & N\\
L10a161\{0,0\} & N\\
L10a162\{0,0\} & N\\
L10a163\{0,0\} & N\\
L10a163\{0,1\} & N\\
L10a164\{1,0\} & N\\
L10a164\{1,1\} & N\\
\end{tabular}%\hspace{.01\textwidth}
\begin{tabular}[t]{|cc}
name & strongly QP (SL)\\
\hline\hline
L10a165\{0,0,0\} & N\\
L10a166\{1,0,0\} & N\\
L10a166\{0,1,1\} & N\\
L10a166\{1,1,1\} & N\\
L10a167\{0,0,0\} & N\\
L10a167\{1,0,0\} & N\\
L10a167\{0,1,1\} & N\\
L10a168\{0,0,0\} & N\\
L10a168\{0,1,0\} & N\\
L10a168\{1,1,0\} & N\\
L10a168\{1,0,1\} & N\\
L10a169\{0,0,0\} & N\\
L10a169\{1,0,0\} & N\\
L10a169\{0,1,0\} & N\\
L10a169\{1,1,0\} & N\\
L10a169\{0,0,1\} & N\\
L10a169\{1,0,1\} & N\\
L10a169\{0,1,1\} & N\\
L10a169\{1,1,1\} & N\\
L10a170\{0,0,0\} & N\\
L10a170\{0,1,0\} & N\\
L10a170\{1,1,0\} & N\\
L10a170\{0,0,1\} & N\\
L10a170\{1,0,1\} & N\\
L10a170\{1,1,1\} & N\\
L10a171\{0,1,0\} & N\\
L10a171\{1,1,0\} & N\\
L10a171\{1,0,1\} & N\\
L10a171\{1,1,1\} & N\\
L10a172\{0,0,0\} & N\\
L10a172\{1,1,0\} & N\\
L10a172\{0,0,1\} & N\\
L10a172\{1,0,1\} & N\\
L10a172\{1,1,1\} & N\\
L10a173\{1,0,0\} & N\\
L10a173\{0,1,0\} & N\\
L10a173\{1,0,1\} & N\\
L10a173\{0,1,1\} & N\\
L10a173\{1,1,1\} & N\\
L10a174\{1,0,0,0\} & N\\
L10a174\{0,1,0,0\} & N\\
L10a174\{0,0,1,0\} & N\\
L10a174\{1,0,1,0\} & N\\
L10a174\{0,0,0,1\} & N\\
L10a174\{0,1,0,1\} & N\\
L10a174\{0,0,1,1\} & N\\
L10a174\{1,0,1,1\} & N\\
L10a174\{0,1,1,1\} & N\\
L10a174\{1,1,1,1\} & N\\
L10n1\{0\} & N\\
L10n1\{1\} & N\\
L10n2\{0\} & N\\
L10n2\{1\} & N\\
L10n3\{0\} & N\\
\end{tabular}%\hspace{.01\textwidth}
\begin{tabular}[t]{|cc}
name & strongly QP (SL)\\
\hline\hline
L10n3\{1\} & N\\
L10n4\{0\} & N\\
L10n4\{1\} & N\\
L10n5\{0\} & N\\
L10n5\{1\} & N\\
L10n6\{0\} & N\\
L10n6\{1\} & N\\
L10n9\{0\} & N\\
L10n9\{1\} & N\\
L10n10\{0\} & N\\
L10n10\{1\} & N\\
L10n11\{0\} & N\\
L10n11\{1\} & N\\
L10n12\{0\} & N\\
L10n12\{1\} & N\\
L10n13\{0\} & N\\
L10n13\{1\} & N\\
L10n15\{0\} & N\\
L10n15\{1\} & N\\
L10n17\{0\} & N\\
L10n17\{1\} & N\\
L10n20\{0\} & N\\
L10n20\{1\} & N\\
L10n21\{0\} & N\\
L10n21\{1\} & N\\
L10n22\{0\} & N\\
L10n22\{1\} & N\\
L10n23\{0\} & N\\
L10n23\{1\} & N\\
L10n24\{0\} & N\\
L10n24\{1\} & N\\
L10n25\{0\} & N\\
L10n25\{1\} & N\\
L10n26\{0\} & N\\
L10n26\{1\} & N\\
L10n27\{0\} & N\\
L10n27\{1\} & N\\
L10n28\{0\} & N\\
L10n28\{1\} & N\\
L10n29\{0\} & N\\
L10n30\{0\} & N\\
L10n33\{0\} & N\\
L10n33\{1\} & N\\
L10n37\{1\} & Y (6)\\
L10n38\{1\} & N\\
L10n39\{0\} & N\\
L10n39\{1\} & N\\
L10n40\{0\} & N\\
L10n40\{1\} & N\\
L10n41\{0\} & N\\
L10n41\{1\} & N\\
L10n42\{0\} & N\\
L10n42\{1\} & N\\
L10n43\{0\} & N\\
\end{tabular}
\end{table}
\begin{table}[p]
\centering
\small
\setlength{\tabcolsep}{4pt}
\begin{tabular}[t]{cc}
name & strongly QP (SL)\\
\hline\hline
L10n43\{1\} & N\\
L10n44\{1\} & N\\
L10n45\{0\} & N\\
L10n45\{1\} & N\\
L10n46\{0\} & N\\
L10n46\{1\} & N\\
L10n47\{0\} & N\\
L10n48\{0\} & N\\
L10n49\{0\} & N\\
L10n49\{1\} & N\\
L10n50\{0\} & N\\
L10n50\{1\} & N\\
L10n51\{0\} & N\\
L10n52\{0\} & N\\
L10n52\{1\} & N\\
L10n53\{0\} & N\\
L10n53\{1\} & N\\
L10n54\{0\} & N\\
L10n55\{0\} & N\\
L10n56\{0\} & N\\
L10n56\{1\} & N\\
L10n57\{0\} & N\\
L10n57\{1\} & N\\
L10n60\{1\} & Y (6)\\
L10n61\{1\} & N\\
L10n62\{0\} & N\\
L10n62\{1\} & N\\
L10n63\{0\} & N\\
L10n64\{0\} & N\\
L10n64\{1\} & N\\
L10n65\{0,0\} & N\\
L10n65\{0,1\} & N\\
L10n65\{1,1\} & N\\
L10n66\{1,0\} & N\\
L10n67\{1,1\} & N\\
L10n68\{1,0\} & N\\
L10n68\{0,1\} & N\\
L10n69\{1,0\} & N\\
L10n69\{0,1\} & N\\
L10n70\{0,0\} & N\\
L10n70\{1,0\} & N\\
L10n70\{0,1\} & N\\
L10n70\{1,1\} & N\\
L10n71\{0,0\} & N\\
L10n71\{1,0\} & N\\
L10n71\{0,1\} & N\\
L10n71\{1,1\} & N\\
L10n72\{0,0\} & N\\
L10n72\{1,0\} & N\\
L10n72\{0,1\} & $\mathrm{Y}^{*}$ (5)\\
L10n72\{1,1\} & N\\
L10n73\{1,0\} & N\\
L10n73\{0,1\} & N\\
L10n74\{1,0\} & N\\
\end{tabular}%\hspace{.01\textwidth}
\begin{tabular}[t]{|cc}
name & strongly QP (SL)\\
\hline\hline
L10n74\{0,1\} & N\\
L10n76\{1,0\} & N\\
L10n76\{0,1\} & N\\
L10n77\{0,0\} & $\mathrm{Y}^{*}$ (7)\\
L10n77\{1,0\} & N\\
L10n77\{0,1\} & N\\
L10n77\{1,1\} & N\\
L10n78\{0,0\} & N\\
L10n78\{1,1\} & N\\
L10n79\{0,0\} & N\\
L10n79\{1,1\} & N\\
L10n80\{0,0\} & N\\
L10n80\{1,1\} & N\\
L10n81\{0,0\} & N\\
L10n81\{0,1\} & N\\
L10n81\{1,1\} & N\\
L10n82\{0,0\} & N\\
L10n82\{1,0\} & Y (7)\\
L10n82\{0,1\} & N\\
L10n82\{1,1\} & N\\
L10n83\{0,0\} & N\\
L10n83\{1,1\} & N\\
L10n84\{0,0\} & $\mathrm{Y}^{*}$ (5)\\
L10n84\{1,1\} & N\\
L10n85\{0,0\} & N\\
L10n85\{1,1\} & N\\
L10n86\{0,0\} & N\\
L10n86\{1,0\} & N\\
L10n86\{1,1\} & N\\
L10n87\{0,0\} & N\\
L10n87\{1,0\} & N\\
L10n87\{1,1\} & Y (7)\\
L10n88\{1,0\} & N\\
L10n88\{0,1\} & $\mathrm{Y}^{*}$ (5)\\
L10n88\{1,1\} & N\\
L10n89\{1,0\} & N\\
L10n89\{0,1\} & N\\
L10n89\{1,1\} & N\\
L10n90\{0,0\} & N\\
L10n90\{1,1\} & N\\
L10n91\{0,0\} & N\\
L10n92\{0,0\} & N\\
L10n92\{1,0\} & N\\
L10n93\{0,0\} & $\mathrm{Y}^{*}$ (7)\\
L10n93\{1,0\} & N\\
L10n93\{1,1\} & N\\
L10n94\{1,0\} & N\\
L10n94\{0,1\} & $\mathrm{Y}^{*}$ (5)\\
L10n94\{1,1\} & N\\
L10n95\{1,0\} & N\\
L10n96\{0,1,1\} & N\\
L10n97\{1,1,0\} & N\\
L10n97\{0,1,1\} & $\mathrm{Y}^{*}$ (6)\\
L10n98\{1,0,0\} & N\\
\end{tabular}%\hspace{.01\textwidth}
\begin{tabular}[t]{|cc}
name & strongly QP (SL)\\
\hline\hline
L10n98\{0,1,1\} & N\\
L10n98\{1,1,1\} & N\\
L10n99\{0,0,0\} & N\\
L10n99\{1,0,0\} & N\\
L10n99\{0,1,0\} & N\\
L10n99\{1,1,1\} & N\\
L10n100\{0,0,0\} & N\\
L10n100\{1,0,0\} & N\\
L10n100\{0,1,0\} & N\\
L10n100\{1,1,0\} & N\\
L10n100\{0,0,1\} & N\\
L10n100\{1,0,1\} & N\\
L10n100\{0,1,1\} & N\\
L10n100\{1,1,1\} & N\\
L10n101\{0,0,0\} & N\\
L10n101\{1,0,0\} & N\\
L10n101\{0,1,0\} & $\mathrm{Y}^{*}$ (6)\\
L10n101\{0,1,1\} & N\\
L10n101\{1,1,1\} & N\\
L10n102\{0,0,0\} & $\mathrm{Y}^{*}$ (4)\\
L10n102\{1,0,0\} & N\\
L10n102\{0,1,0\} & N\\
L10n102\{1,1,0\} & N\\
L10n102\{1,0,1\} & N\\
L10n102\{0,1,1\} & N\\
L10n102\{1,1,1\} & N\\
L10n103\{0,0,0\} & N\\
L10n103\{1,0,0\} & N\\
L10n103\{0,1,0\} & N\\
L10n103\{1,1,0\} & N\\
L10n103\{0,0,1\} & N\\
L10n103\{1,0,1\} & N\\
L10n103\{0,1,1\} & N\\
L10n103\{1,1,1\} & N\\
L10n104\{0,1,0\} & N\\
L10n104\{1,0,1\} & Y (6)\\
L10n104\{1,1,1\} & N\\
L10n105\{0,1,0\} & N\\
L10n105\{0,0,1\} & N\\
L10n105\{1,0,1\} & N\\
L10n105\{0,1,1\} & N\\
L10n106\{0,0,0\} & N\\
L10n106\{1,0,0\} & N\\
L10n106\{0,1,0\} & N\\
L10n106\{1,1,0\} & N\\
L10n106\{0,0,1\} & N\\
L10n106\{1,0,1\} & N\\
L10n106\{0,1,1\} & N\\
L10n106\{1,1,1\} & N\\
L10n108\{0,0,0\} & N\\
L10n108\{0,1,0\} & N\\
L10n108\{1,1,0\} & N\\
L10n108\{0,0,1\} & N\\
L10n108\{1,0,1\} & N\\
\end{tabular}
\end{table}
\begin{table}[p]
\centering
\small
\setlength{\tabcolsep}{4pt}
\begin{tabular}[t]{cc}
name & strongly QP (SL)\\
\hline\hline
L10n108\{1,1,1\} & N\\
L10n109\{0,0,0\} & N\\
L10n109\{0,1,0\} & N\\
L10n109\{1,1,0\} & N\\
L10n109\{0,0,1\} & N\\
L10n109\{1,0,1\} & N\\
L10n109\{1,1,1\} & N\\
L10n110\{0,0,0\} & N\\
L10n110\{0,1,0\} & N\\
L10n110\{0,1,1\} & N\\
L10n110\{1,1,1\} & N\\
L10n111\{0,0,0\} & N\\
L10n111\{1,1,0\} & N\\
L10n111\{1,1,1\} & N\\
L10n112\{0,1,0,0\} & N\\
L10n112\{1,1,0,0\} & N\\
L10n112\{0,0,1,0\} & N\\
L10n112\{1,1,1,0\} & N\\
L10n112\{0,0,0,1\} & N\\
L10n112\{1,0,0,1\} & N\\
L10n112\{0,1,0,1\} & N\\
L10n112\{1,1,0,1\} & N\\
L10n112\{0,0,1,1\} & N\\
L10n112\{0,1,1,1\} & N\\
L10n113\{0,0,0,0\} & N\\
L10n113\{0,1,0,0\} & N\\
L10n113\{1,1,0,0\} & N\\
L10n113\{1,0,1,0\} & Y (5)\\
L10n113\{0,1,1,0\} & N\\
L10n113\{0,0,0,1\} & N\\
L10n113\{1,0,0,1\} & N\\
L10n113\{1,1,0,1\} & N\\
L10n113\{0,0,1,1\} & N\\
L10n113\{0,1,1,1\} & N\\
L10n113\{1,1,1,1\} & N\\
L11a1\{0\} & N\\
L11a1\{1\} & N\\
L11a2\{0\} & N\\
L11a2\{1\} & N\\
L11a3\{0\} & N\\
L11a3\{1\} & N\\
L11a8\{0\} & N\\
L11a8\{1\} & N\\
L11a11\{0\} & N\\
L11a11\{1\} & N\\
L11a12\{0\} & N\\
L11a12\{1\} & N\\
L11a13\{0\} & N\\
L11a13\{1\} & N\\
L11a16\{0\} & N\\
L11a16\{1\} & N\\
L11a18\{0\} & N\\
L11a18\{1\} & N\\
L11a19\{0\} & N\\
\end{tabular}%\hspace{.01\textwidth}
\begin{tabular}[t]{|cc}
name & strongly QP (SL)\\
\hline\hline
L11a19\{1\} & N\\
L11a20\{0\} & N\\
L11a20\{1\} & N\\
L11a24\{1\} & N\\
L11a25\{0\} & N\\
L11a27\{1\} & N\\
L11a28\{0\} & N\\
L11a28\{1\} & N\\
L11a30\{1\} & N\\
L11a31\{0\} & N\\
L11a33\{1\} & N\\
L11a35\{1\} & N\\
L11a38\{0\} & N\\
L11a38\{1\} & N\\
L11a40\{0\} & N\\
L11a40\{1\} & N\\
L11a41\{0\} & N\\
L11a41\{1\} & N\\
L11a42\{0\} & N\\
L11a42\{1\} & N\\
L11a44\{0\} & N\\
L11a44\{1\} & N\\
L11a46\{0\} & N\\
L11a46\{1\} & N\\
L11a48\{0\} & N\\
L11a48\{1\} & N\\
L11a49\{0\} & N\\
L11a49\{1\} & N\\
L11a51\{0\} & N\\
L11a51\{1\} & N\\
L11a55\{0\} & N\\
L11a55\{1\} & N\\
L11a57\{0\} & N\\
L11a57\{1\} & N\\
L11a58\{0\} & N\\
L11a59\{0\} & N\\
L11a62\{1\} & N\\
L11a63\{0\} & N\\
L11a63\{1\} & N\\
L11a66\{0\} & N\\
L11a66\{1\} & N\\
L11a67\{0\} & N\\
L11a69\{0\} & N\\
L11a69\{1\} & N\\
L11a70\{0\} & N\\
L11a71\{0\} & N\\
L11a72\{1\} & N\\
L11a74\{1\} & N\\
L11a77\{0\} & N\\
L11a77\{1\} & N\\
L11a78\{0\} & N\\
L11a78\{1\} & N\\
L11a80\{0\} & N\\
L11a82\{1\} & N\\
\end{tabular}%\hspace{.01\textwidth}
\begin{tabular}[t]{|cc}
name & strongly QP (SL)\\
\hline\hline
L11a85\{0\} & N\\
L11a87\{0\} & N\\
L11a89\{1\} & N\\
L11a90\{0\} & N\\
L11a90\{1\} & N\\
L11a91\{0\} & N\\
L11a91\{1\} & N\\
L11a92\{1\} & N\\
L11a93\{0\} & N\\
L11a93\{1\} & N\\
L11a95\{0\} & N\\
L11a95\{1\} & N\\
L11a97\{1\} & N\\
L11a100\{0\} & N\\
L11a100\{1\} & N\\
L11a104\{0\} & N\\
L11a105\{0\} & N\\
L11a106\{0\} & N\\
L11a107\{1\} & N\\
L11a110\{0\} & N\\
L11a110\{1\} & N\\
L11a113\{0\} & N\\
L11a118\{1\} & N\\
L11a121\{0\} & N\\
L11a123\{0\} & N\\
L11a127\{1\} & N\\
L11a128\{0\} & N\\
L11a133\{1\} & N\\
L11a134\{1\} & N\\
L11a139\{0\} & N\\
L11a139\{1\} & N\\
L11a140\{0\} & N\\
L11a140\{1\} & N\\
L11a141\{0\} & N\\
L11a143\{1\} & N\\
L11a146\{0\} & N\\
L11a147\{0\} & N\\
L11a150\{1\} & N\\
L11a152\{1\} & N\\
L11a155\{1\} & N\\
L11a156\{0\} & N\\
L11a157\{1\} & N\\
L11a158\{0\} & N\\
L11a158\{1\} & N\\
L11a159\{0\} & N\\
L11a159\{1\} & N\\
L11a160\{0\} & N\\
L11a160\{1\} & N\\
L11a161\{0\} & N\\
L11a161\{1\} & N\\
L11a164\{0\} & N\\
L11a166\{0\} & N\\
L11a168\{0\} & N\\
L11a170\{0\} & N\\
\end{tabular}
\end{table}
\begin{table}[p]
\centering
\small
\setlength{\tabcolsep}{4pt}
\begin{tabular}[t]{cc}
name & strongly QP (SL)\\
\hline\hline
L11a171\{1\} & N\\
L11a172\{0\} & N\\
L11a173\{0\} & N\\
L11a176\{0\} & N\\
L11a177\{0\} & N\\
L11a178\{1\} & N\\
L11a179\{1\} & N\\
L11a180\{0\} & N\\
L11a180\{1\} & N\\
L11a182\{1\} & N\\
L11a183\{1\} & N\\
L11a184\{1\} & N\\
L11a185\{1\} & N\\
L11a188\{1\} & N\\
L11a190\{1\} & N\\
L11a193\{1\} & N\\
L11a195\{0\} & N\\
L11a197\{1\} & N\\
L11a198\{0\} & N\\
L11a200\{0\} & N\\
L11a202\{0\} & N\\
L11a207\{1\} & N\\
L11a208\{0\} & N\\
L11a210\{1\} & N\\
L11a213\{0\} & N\\
L11a214\{0\} & N\\
L11a215\{1\} & N\\
L11a216\{1\} & N\\
L11a217\{1\} & N\\
L11a218\{1\} & N\\
L11a219\{0\} & N\\
L11a220\{0\} & N\\
L11a220\{1\} & N\\
L11a223\{0\} & N\\
L11a223\{1\} & N\\
L11a224\{0\} & N\\
L11a226\{1\} & N\\
L11a227\{0\} & N\\
L11a229\{0\} & N\\
L11a232\{0\} & N\\
L11a234\{0\} & N\\
L11a235\{1\} & N\\
L11a237\{0\} & N\\
L11a237\{1\} & N\\
L11a238\{0\} & N\\
L11a238\{1\} & N\\
L11a240\{0\} & N\\
L11a243\{0\} & N\\
L11a244\{0\} & N\\
L11a247\{0\} & N\\
L11a247\{1\} & N\\
L11a248\{0\} & N\\
L11a248\{1\} & N\\
L11a249\{0\} & N\\
\end{tabular}%\hspace{.01\textwidth}
\begin{tabular}[t]{|cc}
name & strongly QP (SL)\\
\hline\hline
L11a249\{1\} & N\\
L11a250\{0\} & N\\
L11a251\{0\} & N\\
L11a252\{0\} & N\\
L11a253\{0\} & N\\
L11a254\{0\} & N\\
L11a255\{0\} & N\\
L11a262\{1\} & N\\
L11a264\{0\} & N\\
L11a265\{1\} & N\\
L11a266\{1\} & N\\
L11a267\{0\} & N\\
L11a267\{1\} & N\\
L11a268\{0\} & N\\
L11a268\{1\} & N\\
L11a274\{0\} & N\\
L11a276\{0\} & N\\
L11a277\{0\} & N\\
L11a279\{0\} & N\\
L11a282\{0\} & N\\
L11a285\{0\} & N\\
L11a287\{0\} & N\\
L11a289\{0\} & N\\
L11a289\{1\} & N\\
L11a297\{1\} & N\\
L11a300\{0\} & N\\
L11a301\{0\} & N\\
L11a303\{0\} & N\\
L11a304\{0\} & N\\
L11a307\{0\} & N\\
L11a310\{0\} & N\\
L11a311\{0\} & N\\
L11a311\{1\} & N\\
L11a312\{0\} & N\\
L11a323\{0\} & N\\
L11a323\{1\} & N\\
L11a325\{0\} & N\\
L11a326\{0\} & N\\
L11a326\{1\} & N\\
L11a327\{0\} & N\\
L11a329\{0\} & N\\
L11a330\{0\} & N\\
L11a330\{1\} & N\\
L11a331\{0\} & N\\
L11a333\{0\} & N\\
L11a334\{0\} & N\\
L11a336\{1\} & N\\
L11a339\{0\} & N\\
L11a339\{1\} & N\\
L11a341\{0\} & N\\
L11a341\{1\} & N\\
L11a342\{1\} & N\\
L11a343\{0\} & N\\
L11a343\{1\} & N\\
\end{tabular}%\hspace{.01\textwidth}
\begin{tabular}[t]{|cc}
name & strongly QP (SL)\\
\hline\hline
L11a344\{0\} & N\\
L11a345\{0\} & N\\
L11a346\{0\} & N\\
L11a346\{1\} & N\\
L11a348\{1\} & N\\
L11a349\{0\} & N\\
L11a349\{1\} & N\\
L11a350\{0\} & N\\
L11a351\{0\} & N\\
L11a352\{1\} & N\\
L11a353\{0\} & N\\
L11a354\{0\} & N\\
L11a354\{1\} & N\\
L11a360\{0\} & N\\
L11a364\{0\} & N\\
L11a365\{0\} & N\\
L11a367\{0\} & N\\
L11a373\{1\} & N\\
L11a374\{0\} & N\\
L11a375\{0\} & N\\
L11a377\{0\} & N\\
L11a379\{0\} & N\\
L11a380\{0\} & N\\
L11a381\{0\} & N\\
L11a382\{0\} & N\\
L11a383\{0\} & N\\
L11a386\{0,0\} & N\\
L11a387\{1,0\} & N\\
L11a387\{0,1\} & N\\
L11a387\{1,1\} & N\\
L11a389\{0,0\} & N\\
L11a389\{1,0\} & N\\
L11a389\{0,1\} & N\\
L11a390\{0,0\} & N\\
L11a391\{1,0\} & N\\
L11a391\{0,1\} & N\\
L11a391\{1,1\} & N\\
L11a393\{0,0\} & N\\
L11a393\{1,0\} & N\\
L11a393\{0,1\} & N\\
L11a395\{0,0\} & N\\
L11a398\{0,0\} & N\\
L11a398\{1,1\} & N\\
L11a400\{0,0\} & N\\
L11a400\{1,0\} & N\\
L11a400\{0,1\} & N\\
L11a400\{1,1\} & N\\
L11a401\{0,0\} & N\\
L11a401\{1,0\} & N\\
L11a401\{0,1\} & N\\
L11a401\{1,1\} & N\\
L11a402\{0,0\} & N\\
L11a402\{1,0\} & N\\
L11a402\{1,1\} & N\\
\end{tabular}
\end{table}
\begin{table}[p]
\centering
\small
\setlength{\tabcolsep}{4pt}
\begin{tabular}[t]{cc}
name & strongly QP (SL)\\
\hline\hline
L11a403\{0,0\} & N\\
L11a403\{1,0\} & N\\
L11a403\{1,1\} & N\\
L11a404\{1,1\} & N\\
L11a405\{1,0\} & N\\
L11a405\{0,1\} & N\\
L11a405\{1,1\} & N\\
L11a407\{1,0\} & N\\
L11a407\{0,1\} & N\\
L11a408\{0,0\} & N\\
L11a409\{0,0\} & N\\
L11a409\{1,1\} & N\\
L11a410\{0,0\} & N\\
L11a410\{1,0\} & N\\
L11a410\{1,1\} & N\\
L11a412\{1,0\} & N\\
L11a412\{0,1\} & N\\
L11a412\{1,1\} & N\\
L11a413\{1,1\} & N\\
L11a414\{0,0\} & N\\
L11a415\{1,0\} & N\\
L11a415\{0,1\} & N\\
L11a416\{0,0\} & N\\
L11a419\{0,1\} & N\\
L11a420\{0,0\} & N\\
L11a420\{1,1\} & N\\
L11a421\{1,0\} & N\\
L11a421\{0,1\} & N\\
L11a422\{1,0\} & N\\
L11a422\{0,1\} & N\\
L11a422\{1,1\} & N\\
L11a423\{1,0\} & N\\
L11a423\{0,1\} & N\\
L11a424\{0,0\} & N\\
L11a425\{0,1\} & N\\
L11a426\{0,0\} & N\\
L11a426\{1,1\} & N\\
L11a428\{0,0\} & N\\
L11a428\{0,1\} & N\\
L11a428\{1,1\} & N\\
L11a430\{0,0\} & N\\
L11a430\{1,1\} & N\\
L11a431\{0,0\} & N\\
L11a431\{1,1\} & N\\
L11a432\{0,0\} & N\\
L11a432\{1,1\} & N\\
L11a433\{1,0\} & N\\
L11a433\{0,1\} & N\\
L11a434\{1,0\} & N\\
L11a434\{0,1\} & N\\
L11a435\{0,0\} & N\\
L11a435\{1,0\} & N\\
L11a436\{1,0\} & N\\
L11a436\{1,1\} & N\\
\end{tabular}%\hspace{.01\textwidth}
\begin{tabular}[t]{|cc}
name & strongly QP (SL)\\
\hline\hline
L11a438\{0,0\} & N\\
L11a438\{1,1\} & N\\
L11a439\{0,0\} & N\\
L11a439\{1,1\} & N\\
L11a440\{1,1\} & N\\
L11a441\{1,1\} & N\\
L11a442\{0,0\} & N\\
L11a445\{1,0\} & N\\
L11a446\{0,0\} & N\\
L11a446\{1,1\} & N\\
L11a447\{0,0\} & N\\
L11a447\{1,1\} & N\\
L11a448\{0,0\} & N\\
L11a448\{1,1\} & N\\
L11a449\{0,0\} & N\\
L11a449\{1,1\} & N\\
L11a450\{1,0\} & N\\
L11a451\{1,0\} & N\\
L11a451\{0,1\} & N\\
L11a452\{1,0\} & N\\
L11a453\{0,0\} & N\\
L11a453\{1,1\} & N\\
L11a454\{0,0\} & N\\
L11a454\{1,0\} & N\\
L11a454\{0,1\} & N\\
L11a455\{1,0\} & N\\
L11a455\{0,1\} & N\\
L11a456\{1,0\} & N\\
L11a457\{0,0\} & N\\
L11a457\{1,0\} & N\\
L11a457\{0,1\} & N\\
L11a458\{0,0\} & N\\
L11a458\{1,1\} & N\\
L11a459\{0,0\} & N\\
L11a459\{1,1\} & N\\
L11a460\{1,0\} & N\\
L11a461\{0,0\} & N\\
L11a461\{1,1\} & N\\
L11a462\{0,0\} & N\\
L11a462\{1,1\} & N\\
L11a463\{1,0\} & N\\
L11a466\{1,0\} & N\\
L11a466\{0,1\} & N\\
L11a467\{0,0\} & N\\
L11a467\{1,0\} & N\\
L11a468\{1,1\} & N\\
L11a469\{1,0\} & N\\
L11a470\{0,1\} & N\\
L11a470\{1,1\} & N\\
L11a471\{1,0\} & N\\
L11a471\{0,1\} & N\\
L11a472\{0,0\} & N\\
L11a472\{1,1\} & N\\
L11a473\{0,0\} & N\\
\end{tabular}%\hspace{.01\textwidth}
\begin{tabular}[t]{|cc}
name & strongly QP (SL)\\
\hline\hline
L11a473\{1,1\} & N\\
L11a474\{0,0\} & N\\
L11a474\{1,1\} & N\\
L11a475\{0,0\} & N\\
L11a475\{0,1\} & N\\
L11a476\{0,0\} & N\\
L11a476\{1,1\} & N\\
L11a477\{0,0\} & N\\
L11a477\{1,1\} & N\\
L11a479\{1,0\} & N\\
L11a479\{0,1\} & N\\
L11a480\{0,0\} & N\\
L11a480\{1,1\} & N\\
L11a481\{0,1\} & N\\
L11a483\{1,0\} & N\\
L11a483\{0,1\} & N\\
L11a484\{0,0\} & N\\
L11a484\{1,0\} & N\\
L11a484\{0,1\} & N\\
L11a484\{1,1\} & N\\
L11a485\{0,0\} & N\\
L11a485\{1,0\} & N\\
L11a485\{0,1\} & N\\
L11a485\{1,1\} & N\\
L11a486\{1,0\} & N\\
L11a486\{0,1\} & N\\
L11a487\{0,0\} & N\\
L11a487\{1,0\} & N\\
L11a487\{0,1\} & N\\
L11a487\{1,1\} & N\\
L11a488\{1,0\} & N\\
L11a488\{0,1\} & N\\
L11a489\{1,1\} & N\\
L11a490\{0,0\} & N\\
L11a490\{1,0\} & N\\
L11a491\{1,0\} & N\\
L11a491\{1,1\} & N\\
L11a492\{0,0\} & N\\
L11a492\{1,0\} & N\\
L11a492\{0,1\} & N\\
L11a492\{1,1\} & N\\
L11a493\{0,0\} & N\\
L11a493\{1,1\} & N\\
L11a494\{0,0\} & N\\
L11a494\{1,0\} & N\\
L11a494\{0,1\} & N\\
L11a494\{1,1\} & N\\
L11a496\{0,0\} & N\\
L11a496\{1,0\} & N\\
L11a496\{0,1\} & N\\
L11a496\{1,1\} & N\\
L11a497\{0,0\} & N\\
L11a497\{1,1\} & N\\
L11a498\{1,1\} & N\\
\end{tabular}
\end{table}
\begin{table}[p]
\centering
\small
\setlength{\tabcolsep}{4pt}
\begin{tabular}[t]{cc}
name & strongly QP (SL)\\
\hline\hline
L11a500\{1,0\} & N\\
L11a501\{1,0\} & N\\
L11a502\{1,0\} & N\\
L11a502\{1,1\} & N\\
L11a503\{0,0\} & N\\
L11a503\{0,1\} & N\\
L11a504\{0,1\} & N\\
L11a505\{0,0\} & N\\
L11a505\{1,1\} & N\\
L11a506\{0,0\} & N\\
L11a506\{1,0\} & N\\
L11a506\{0,1\} & N\\
L11a507\{1,0\} & N\\
L11a507\{0,1\} & N\\
L11a508\{1,0\} & N\\
L11a508\{0,1\} & N\\
L11a509\{0,0\} & N\\
L11a509\{1,0\} & N\\
L11a509\{0,1\} & N\\
L11a510\{0,0\} & N\\
L11a510\{1,1\} & N\\
L11a511\{1,0\} & N\\
L11a512\{1,1\} & N\\
L11a513\{0,0\} & N\\
L11a514\{1,0\} & N\\
L11a515\{0,0\} & N\\
L11a515\{1,0\} & N\\
L11a516\{1,0\} & N\\
L11a516\{1,1\} & N\\
L11a517\{0,0\} & N\\
L11a517\{1,0\} & N\\
L11a517\{1,1\} & N\\
L11a519\{0,0\} & N\\
L11a519\{1,0\} & N\\
L11a520\{0,0\} & N\\
L11a520\{1,0\} & N\\
L11a520\{0,1\} & N\\
L11a520\{1,1\} & N\\
L11a521\{0,0\} & N\\
L11a521\{1,0\} & N\\
L11a522\{0,0\} & N\\
L11a522\{1,0\} & N\\
L11a523\{0,0\} & N\\
L11a524\{0,0\} & N\\
L11a524\{1,0\} & N\\
L11a525\{0,1\} & N\\
L11a525\{1,1\} & N\\
L11a526\{1,0\} & N\\
L11a526\{0,1\} & N\\
L11a526\{1,1\} & N\\
L11a527\{1,1\} & N\\
L11a528\{1,0\} & N\\
L11a528\{1,1\} & N\\
L11a529\{1,0\} & N\\
\end{tabular}%\hspace{.01\textwidth}
\begin{tabular}[t]{|cc}
name & strongly QP (SL)\\
\hline\hline
L11a529\{0,1\} & N\\
L11a529\{1,1\} & N\\
L11a530\{0,0\} & N\\
L11a530\{1,0\} & N\\
L11a531\{1,0,0\} & N\\
L11a531\{0,1,0\} & N\\
L11a531\{0,0,1\} & N\\
L11a531\{1,0,1\} & N\\
L11a531\{0,1,1\} & N\\
L11a531\{1,1,1\} & N\\
L11a532\{0,0,0\} & N\\
L11a532\{0,1,0\} & N\\
L11a532\{0,0,1\} & N\\
L11a532\{1,0,1\} & N\\
L11a534\{1,0,0\} & N\\
L11a534\{0,1,0\} & N\\
L11a534\{0,0,1\} & N\\
L11a534\{1,0,1\} & N\\
L11a534\{0,1,1\} & N\\
L11a534\{1,1,1\} & N\\
L11a535\{0,0,0\} & N\\
L11a536\{0,1,0\} & N\\
L11a536\{0,0,1\} & N\\
L11a536\{1,0,1\} & N\\
L11a537\{0,0,0\} & N\\
L11a537\{1,0,0\} & N\\
L11a537\{0,1,0\} & N\\
L11a537\{1,1,0\} & N\\
L11a537\{1,0,1\} & N\\
L11a537\{0,1,1\} & N\\
L11a538\{1,0,0\} & N\\
L11a538\{0,1,0\} & N\\
L11a538\{1,1,0\} & N\\
L11a538\{1,0,1\} & N\\
L11a538\{0,1,1\} & N\\
L11a538\{1,1,1\} & N\\
L11a539\{0,0,0\} & N\\
L11a539\{1,0,0\} & N\\
L11a539\{0,1,0\} & N\\
L11a539\{0,0,1\} & N\\
L11a539\{1,0,1\} & N\\
L11a539\{0,1,1\} & N\\
L11a540\{0,0,0\} & N\\
L11a540\{1,0,0\} & N\\
L11a540\{0,1,1\} & N\\
L11a541\{0,0,0\} & N\\
L11a541\{1,0,0\} & N\\
L11a541\{0,1,0\} & N\\
L11a541\{0,0,1\} & N\\
L11a541\{1,0,1\} & N\\
L11a541\{0,1,1\} & N\\
L11a542\{1,0,0\} & N\\
L11a542\{0,1,0\} & N\\
L11a542\{1,1,0\} & N\\
\end{tabular}%\hspace{.01\textwidth}
\begin{tabular}[t]{|cc}
name & strongly QP (SL)\\
\hline\hline
L11a542\{0,0,1\} & N\\
L11a542\{1,0,1\} & N\\
L11a542\{0,1,1\} & N\\
L11a543\{1,0,0\} & N\\
L11a543\{0,1,0\} & N\\
L11a543\{1,1,0\} & N\\
L11a543\{1,0,1\} & N\\
L11a543\{0,1,1\} & N\\
L11a544\{1,0,0\} & N\\
L11a544\{1,1,0\} & N\\
L11a544\{0,0,1\} & N\\
L11a544\{1,0,1\} & N\\
L11a545\{0,0,0\} & N\\
L11a545\{0,1,0\} & N\\
L11a545\{0,1,1\} & N\\
L11a545\{1,1,1\} & N\\
L11a546\{1,0,0\} & N\\
L11a546\{0,1,0\} & N\\
L11a546\{0,0,1\} & N\\
L11a546\{1,0,1\} & N\\
L11a546\{0,1,1\} & N\\
L11a546\{1,1,1\} & N\\
L11a547\{1,0,0\} & N\\
L11a547\{0,1,0\} & N\\
L11a547\{0,0,1\} & N\\
L11a547\{1,0,1\} & N\\
L11a547\{0,1,1\} & N\\
L11a548\{0,0,0,0\} & N\\
L11a548\{0,0,1,0\} & N\\
L11a548\{1,0,1,0\} & N\\
L11a548\{0,0,0,1\} & N\\
L11a548\{0,1,0,1\} & N\\
L11n1\{0\} & N\\
L11n1\{1\} & N\\
L11n2\{0\} & N\\
L11n2\{1\} & N\\
L11n3\{0\} & N\\
L11n3\{1\} & N\\
L11n4\{0\} & N\\
L11n4\{1\} & N\\
L11n5\{0\} & N\\
L11n5\{1\} & N\\
L11n6\{0\} & N\\
L11n6\{1\} & N\\
L11n7\{0\} & N\\
L11n7\{1\} & N\\
L11n8\{0\} & N\\
L11n8\{1\} & N\\
L11n9\{0\} & N\\
L11n9\{1\} & N\\
L11n12\{0\} & N\\
L11n12\{1\} & N\\
L11n15\{0\} & N\\
L11n15\{1\} & N\\
\end{tabular}
\end{table}
\begin{table}[p]
\centering
\small
\setlength{\tabcolsep}{4pt}
\begin{tabular}[t]{cc}
name & strongly QP (SL)\\
\hline\hline
L11n22\{0\} & N\\
L11n22\{1\} & N\\
L11n23\{0\} & N\\
L11n23\{1\} & N\\
L11n24\{0\} & N\\
L11n24\{1\} & N\\
L11n28\{0\} & N\\
L11n28\{1\} & N\\
L11n32\{0\} & N\\
L11n32\{1\} & N\\
L11n34\{0\} & N\\
L11n34\{1\} & N\\
L11n36\{0\} & N\\
L11n36\{1\} & N\\
L11n38\{0\} & N\\
L11n38\{1\} & N\\
L11n41\{0\} & N\\
L11n41\{1\} & N\\
L11n44\{0\} & $\mathrm{Y}^{*}$ (8)\\
L11n44\{1\} & N\\
L11n45\{0\} & N\\
L11n45\{1\} & N\\
L11n46\{0\} & N\\
L11n46\{1\} & N\\
L11n48\{0\} & N\\
L11n48\{1\} & N\\
L11n49\{0\} & N\\
L11n49\{1\} & N\\
L11n50\{0\} & N\\
L11n50\{1\} & N\\
L11n51\{0\} & N\\
L11n51\{1\} & N\\
L11n52\{0\} & N\\
L11n52\{1\} & N\\
L11n53\{0\} & N\\
L11n53\{1\} & N\\
L11n56\{1\} & N\\
L11n57\{1\} & N\\
L11n58\{1\} & N\\
L11n59\{1\} & Y (6)\\
L11n60\{1\} & N\\
L11n61\{1\} & N\\
L11n62\{0\} & $\mathrm{Y}^{*}$ (6)\\
L11n62\{1\} & N\\
L11n63\{0\} & N\\
L11n63\{1\} & N\\
L11n64\{0\} & N\\
L11n64\{1\} & N\\
L11n65\{1\} & N\\
L11n66\{0\} & N\\
L11n67\{1\} & N\\
L11n68\{0\} & $\mathrm{Y}^{*}$ (6)\\
L11n68\{1\} & N\\
L11n69\{0\} & N\\
\end{tabular}%\hspace{.01\textwidth}
\begin{tabular}[t]{|cc}
name & strongly QP (SL)\\
\hline\hline
L11n69\{1\} & N\\
L11n70\{0\} & N\\
L11n70\{1\} & N\\
L11n74\{0\} & N\\
L11n74\{1\} & Y (8)\\
L11n76\{0\} & N\\
L11n77\{0\} & N\\
L11n77\{1\} & N\\
L11n79\{0\} & N\\
L11n79\{1\} & N\\
L11n80\{0\} & N\\
L11n80\{1\} & N\\
L11n81\{0\} & N\\
L11n81\{1\} & N\\
L11n82\{1\} & N\\
L11n84\{0\} & N\\
L11n84\{1\} & N\\
L11n85\{0\} & N\\
L11n86\{0\} & N\\
L11n87\{0\} & N\\
L11n87\{1\} & N\\
L11n88\{0\} & N\\
L11n88\{1\} & N\\
L11n89\{0\} & N\\
L11n91\{0\} & N\\
L11n91\{1\} & N\\
L11n94\{0\} & N\\
L11n94\{1\} & N\\
L11n95\{0\} & $\mathrm{Y}^{*}$ (6)\\
L11n95\{1\} & N\\
L11n96\{0\} & N\\
L11n96\{1\} & N\\
L11n97\{0\} & N\\
L11n97\{1\} & N\\
L11n98\{0\} & N\\
L11n98\{1\} & Y (8)\\
L11n100\{1\} & N\\
L11n103\{1\} & N\\
L11n104\{0\} & N\\
L11n105\{0\} & N\\
L11n105\{1\} & N\\
L11n106\{0\} & N\\
L11n109\{1\} & N\\
L11n110\{0\} & N\\
L11n111\{0\} & N\\
L11n114\{0\} & N\\
L11n114\{1\} & N\\
L11n115\{1\} & N\\
L11n116\{1\} & N\\
L11n117\{1\} & N\\
L11n120\{1\} & N\\
L11n121\{1\} & Y (6)\\
L11n122\{1\} & Y (4)\\
L11n123\{1\} & N\\
\end{tabular}%\hspace{.01\textwidth}
\begin{tabular}[t]{|cc}
name & strongly QP (SL)\\
\hline\hline
L11n125\{1\} & N\\
L11n126\{0\} & N\\
L11n126\{1\} & N\\
L11n127\{0\} & N\\
L11n127\{1\} & N\\
L11n128\{0\} & N\\
L11n128\{1\} & N\\
L11n129\{0\} & N\\
L11n129\{1\} & N\\
L11n130\{0\} & N\\
L11n131\{0\} & N\\
L11n131\{1\} & N\\
L11n132\{1\} & N\\
L11n133\{0\} & $\mathrm{Y}^{*}$ (8)\\
L11n133\{1\} & N\\
L11n134\{0\} & N\\
L11n135\{0\} & N\\
L11n135\{1\} & N\\
L11n136\{0\} & N\\
L11n137\{0\} & N\\
L11n138\{0\} & N\\
L11n139\{0\} & N\\
L11n140\{0\} & N\\
L11n145\{0\} & N\\
L11n146\{1\} & N\\
L11n147\{0\} & N\\
L11n148\{0\} & N\\
L11n148\{1\} & Y (6)\\
L11n149\{0\} & N\\
L11n149\{1\} & N\\
L11n150\{0\} & N\\
L11n152\{0\} & N\\
L11n152\{1\} & N\\
L11n154\{0\} & N\\
L11n155\{0\} & N\\
L11n157\{0\} & N\\
L11n158\{0\} & N\\
L11n158\{1\} & N\\
L11n159\{0\} & N\\
L11n159\{1\} & N\\
L11n160\{1\} & N\\
L11n161\{1\} & N\\
L11n162\{1\} & Y (6)\\
L11n163\{1\} & N\\
L11n164\{1\} & N\\
L11n165\{1\} & N\\
L11n166\{0\} & N\\
L11n167\{0\} & N\\
L11n168\{0\} & N\\
L11n169\{0\} & $\mathrm{Y}^{*}$ (8)\\
L11n169\{1\} & N\\
L11n170\{0\} & N\\
L11n170\{1\} & N\\
L11n171\{0\} & N\\
\end{tabular}
\end{table}
\begin{table}[p]
\centering
\small
\setlength{\tabcolsep}{4pt}
\begin{tabular}[t]{cc}
name & strongly QP (SL)\\
\hline\hline
L11n172\{0\} & N\\
L11n173\{0\} & N\\
L11n174\{0\} & N\\
L11n175\{1\} & N\\
L11n176\{1\} & N\\
L11n177\{0\} & N\\
L11n177\{1\} & N\\
L11n178\{0\} & N\\
L11n179\{1\} & N\\
L11n180\{0\} & N\\
L11n180\{1\} & Y (4)\\
L11n181\{1\} & Y (6)\\
L11n183\{1\} & N\\
L11n184\{1\} & N\\
L11n185\{1\} & N\\
L11n186\{0\} & N\\
L11n187\{1\} & N\\
L11n188\{1\} & N\\
L11n189\{0\} & N\\
L11n189\{1\} & Y (6)\\
L11n190\{0\} & N\\
L11n190\{1\} & N\\
L11n191\{0\} & N\\
L11n191\{1\} & N\\
L11n192\{1\} & N\\
L11n193\{1\} & N\\
L11n194\{1\} & N\\
L11n195\{0\} & N\\
L11n196\{1\} & N\\
L11n197\{0\} & N\\
L11n197\{1\} & N\\
L11n198\{1\} & N\\
L11n199\{1\} & N\\
L11n202\{0\} & N\\
L11n203\{0\} & N\\
L11n204\{0\} & $\mathrm{Y}^{*}$ (8)\\
L11n204\{1\} & N\\
L11n205\{0\} & N\\
L11n205\{1\} & Y (6)\\
L11n206\{0\} & N\\
L11n207\{0\} & N\\
L11n207\{1\} & N\\
L11n208\{0\} & N\\
L11n208\{1\} & N\\
L11n209\{0\} & N\\
L11n210\{0\} & N\\
L11n211\{0\} & N\\
L11n212\{0\} & N\\
L11n213\{0\} & N\\
L11n214\{0\} & N\\
L11n214\{1\} & N\\
L11n218\{0\} & N\\
L11n218\{1\} & N\\
L11n219\{0\} & N\\
\end{tabular}%\hspace{.01\textwidth}
\begin{tabular}[t]{|cc}
name & strongly QP (SL)\\
\hline\hline
L11n219\{1\} & N\\
L11n220\{1\} & N\\
L11n222\{0\} & $\mathrm{Y}^{*}$ (8)\\
L11n224\{0\} & $\mathrm{Y}^{*}$ (6)\\
L11n224\{1\} & N\\
L11n226\{0\} & N\\
L11n226\{1\} & N\\
L11n227\{0\} & N\\
L11n228\{0\} & N\\
L11n228\{1\} & N\\
L11n231\{0\} & N\\
L11n231\{1\} & N\\
L11n232\{0\} & N\\
L11n232\{1\} & N\\
L11n233\{0\} & $\mathrm{Y}^{*}$ (8)\\
L11n233\{1\} & N\\
L11n234\{0\} & N\\
L11n234\{1\} & N\\
L11n235\{0\} & N\\
L11n235\{1\} & Y (4)\\
L11n236\{1\} & N\\
L11n237\{0\} & N\\
L11n238\{0\} & N\\
L11n239\{1\} & N\\
L11n241\{0\} & N\\
L11n241\{1\} & N\\
L11n243\{1\} & N\\
L11n245\{1\} & N\\
L11n246\{1\} & N\\
L11n248\{0\} & N\\
L11n248\{1\} & N\\
L11n249\{0\} & N\\
L11n250\{0\} & N\\
L11n251\{0\} & N\\
L11n251\{1\} & N\\
L11n252\{1\} & Y (6)\\
L11n253\{1\} & Y (4)\\
L11n254\{0\} & $\mathrm{Y}^{*}$ (8)\\
L11n256\{1,1\} & N\\
L11n258\{0,0\} & N\\
L11n258\{1,0\} & N\\
L11n258\{1,1\} & N\\
L11n259\{0,0\} & N\\
L11n259\{1,0\} & N\\
L11n259\{0,1\} & N\\
L11n259\{1,1\} & N\\
L11n260\{0,0\} & N\\
L11n260\{1,0\} & N\\
L11n260\{0,1\} & N\\
L11n260\{1,1\} & N\\
L11n261\{0,0\} & N\\
L11n261\{1,0\} & N\\
L11n261\{0,1\} & N\\
L11n261\{1,1\} & N\\
\end{tabular}%\hspace{.01\textwidth}
\begin{tabular}[t]{|cc}
name & strongly QP (SL)\\
\hline\hline
L11n262\{0,0\} & $\mathrm{Y}^{*}$ (5)\\
L11n262\{1,0\} & N\\
L11n262\{0,1\} & N\\
L11n262\{1,1\} & N\\
L11n263\{0,0\} & N\\
L11n263\{1,0\} & N\\
L11n263\{0,1\} & N\\
L11n263\{1,1\} & N\\
L11n264\{0,0\} & N\\
L11n264\{1,0\} & N\\
L11n264\{0,1\} & N\\
L11n264\{1,1\} & N\\
L11n265\{0,1\} & N\\
L11n266\{0,1\} & $\mathrm{Y}^{*}$ (7)\\
L11n267\{1,0\} & N\\
L11n267\{0,1\} & N\\
L11n267\{1,1\} & N\\
L11n268\{0,1\} & N\\
L11n269\{0,1\} & $\mathrm{Y}^{*}$ (7)\\
L11n270\{1,0\} & N\\
L11n270\{0,1\} & N\\
L11n270\{1,1\} & N\\
L11n274\{0,0\} & N\\
L11n274\{1,1\} & N\\
L11n275\{0,0\} & N\\
L11n275\{1,1\} & N\\
L11n276\{0,0\} & $\mathrm{Y}^{*}$ (7)\\
L11n276\{1,1\} & N\\
L11n277\{0,0\} & N\\
L11n277\{1,0\} & N\\
L11n277\{1,1\} & N\\
L11n278\{0,0\} & N\\
L11n279\{1,0\} & N\\
L11n280\{0,0\} & N\\
L11n280\{1,0\} & N\\
L11n280\{0,1\} & N\\
L11n280\{1,1\} & N\\
L11n281\{0,0\} & N\\
L11n281\{1,0\} & N\\
L11n281\{0,1\} & N\\
L11n281\{1,1\} & N\\
L11n282\{0,0\} & N\\
L11n282\{1,0\} & N\\
L11n282\{0,1\} & N\\
L11n282\{1,1\} & N\\
L11n283\{0,0\} & N\\
L11n283\{1,0\} & N\\
L11n283\{0,1\} & N\\
L11n283\{1,1\} & N\\
L11n284\{0,0\} & N\\
L11n284\{1,0\} & N\\
L11n284\{0,1\} & N\\
L11n284\{1,1\} & N\\
L11n287\{0,0\} & N\\
\end{tabular}
\end{table}
\begin{table}[p]
\centering
\small
\setlength{\tabcolsep}{4pt}
\begin{tabular}[t]{cc}
name & strongly QP (SL)\\
\hline\hline
L11n287\{1,0\} & N\\
L11n287\{0,1\} & N\\
L11n287\{1,1\} & N\\
L11n288\{0,0\} & N\\
L11n288\{1,0\} & N\\
L11n288\{0,1\} & N\\
L11n288\{1,1\} & N\\
L11n289\{0,0\} & N\\
L11n289\{1,0\} & N\\
L11n289\{0,1\} & N\\
L11n289\{1,1\} & N\\
L11n290\{0,0\} & N\\
L11n290\{1,0\} & N\\
L11n290\{0,1\} & N\\
L11n290\{1,1\} & N\\
L11n291\{0,0\} & N\\
L11n291\{1,0\} & N\\
L11n291\{0,1\} & N\\
L11n291\{1,1\} & N\\
L11n292\{0,0\} & N\\
L11n292\{1,0\} & N\\
L11n292\{0,1\} & N\\
L11n292\{1,1\} & N\\
L11n294\{0,0\} & N\\
L11n294\{1,0\} & N\\
L11n294\{1,1\} & N\\
L11n295\{0,0\} & N\\
L11n295\{1,0\} & N\\
L11n295\{1,1\} & N\\
L11n296\{0,0\} & N\\
L11n296\{1,0\} & N\\
L11n296\{1,1\} & N\\
L11n297\{0,0\} & N\\
L11n297\{1,0\} & N\\
L11n297\{1,1\} & N\\
L11n298\{0,0\} & N\\
L11n298\{0,1\} & N\\
L11n298\{1,1\} & N\\
L11n299\{0,0\} & N\\
L11n299\{1,1\} & N\\
L11n300\{1,0\} & N\\
L11n300\{0,1\} & N\\
L11n303\{1,0\} & N\\
L11n303\{1,1\} & Y (7)\\
L11n304\{1,1\} & N\\
L11n305\{1,1\} & N\\
L11n306\{0,0\} & N\\
L11n306\{1,1\} & N\\
L11n307\{0,1\} & N\\
L11n309\{1,0\} & N\\
L11n309\{0,1\} & N\\
L11n309\{1,1\} & N\\
L11n310\{0,0\} & N\\
L11n310\{0,1\} & N\\
\end{tabular}%\hspace{.01\textwidth}
\begin{tabular}[t]{|cc}
name & strongly QP (SL)\\
\hline\hline
L11n310\{1,1\} & N\\
L11n311\{0,0\} & N\\
L11n311\{1,1\} & N\\
L11n312\{0,1\} & N\\
L11n313\{1,1\} & N\\
L11n314\{1,0\} & N\\
L11n314\{0,1\} & N\\
L11n315\{1,0\} & N\\
L11n315\{0,1\} & N\\
L11n316\{1,0\} & N\\
L11n316\{0,1\} & N\\
L11n317\{0,0\} & N\\
L11n317\{1,0\} & N\\
L11n317\{0,1\} & N\\
L11n317\{1,1\} & N\\
L11n318\{0,1\} & N\\
L11n319\{0,0\} & N\\
L11n319\{0,1\} & N\\
L11n319\{1,1\} & N\\
L11n320\{0,0\} & N\\
L11n320\{1,0\} & N\\
L11n320\{0,1\} & N\\
L11n320\{1,1\} & N\\
L11n321\{0,0\} & N\\
L11n321\{1,1\} & N\\
L11n322\{0,0\} & N\\
L11n322\{1,0\} & N\\
L11n322\{0,1\} & N\\
L11n322\{1,1\} & N\\
L11n323\{0,0\} & N\\
L11n323\{1,0\} & N\\
L11n323\{0,1\} & N\\
L11n323\{1,1\} & N\\
L11n324\{0,0\} & N\\
L11n324\{1,1\} & N\\
L11n325\{0,0\} & N\\
L11n325\{1,0\} & N\\
L11n325\{0,1\} & N\\
L11n325\{1,1\} & N\\
L11n326\{0,0\} & N\\
L11n326\{1,1\} & N\\
L11n327\{0,0\} & N\\
L11n327\{1,1\} & N\\
L11n328\{0,0\} & N\\
L11n328\{1,1\} & N\\
L11n329\{1,0\} & N\\
L11n329\{0,1\} & N\\
L11n330\{1,0\} & N\\
L11n330\{0,1\} & N\\
L11n331\{1,0\} & N\\
L11n331\{0,1\} & N\\
L11n332\{1,0\} & N\\
L11n332\{0,1\} & N\\
L11n333\{1,0\} & N\\
\end{tabular}%\hspace{.01\textwidth}
\begin{tabular}[t]{|cc}
name & strongly QP (SL)\\
\hline\hline
L11n333\{0,1\} & N\\
L11n335\{0,0\} & N\\
L11n335\{1,1\} & N\\
L11n336\{1,0\} & N\\
L11n336\{0,1\} & N\\
L11n337\{0,0\} & $\mathrm{Y}^{*}$ (5)\\
L11n337\{1,0\} & N\\
L11n337\{0,1\} & N\\
L11n338\{0,0\} & N\\
L11n338\{1,0\} & Y (7)\\
L11n338\{0,1\} & N\\
L11n339\{0,0\} & N\\
L11n339\{0,1\} & N\\
L11n340\{0,0\} & N\\
L11n340\{0,1\} & N\\
L11n341\{0,0\} & N\\
L11n341\{1,0\} & N\\
L11n342\{0,0\} & N\\
L11n342\{1,0\} & N\\
L11n342\{0,1\} & N\\
L11n342\{1,1\} & N\\
L11n344\{0,0\} & N\\
L11n344\{1,1\} & N\\
L11n345\{0,0\} & N\\
L11n345\{0,1\} & $\mathrm{Y}^{*}$ (3)\\
L11n345\{1,1\} & N\\
L11n346\{0,0\} & N\\
L11n346\{0,1\} & N\\
L11n346\{1,1\} & N\\
L11n347\{0,0\} & N\\
L11n347\{1,0\} & Y (5)\\
L11n347\{0,1\} & N\\
L11n347\{1,1\} & N\\
L11n348\{0,0\} & N\\
L11n348\{1,1\} & N\\
L11n349\{0,0\} & N\\
L11n349\{1,0\} & N\\
L11n349\{0,1\} & N\\
L11n350\{1,0\} & N\\
L11n350\{0,1\} & N\\
L11n351\{0,0\} & N\\
L11n351\{0,1\} & N\\
L11n351\{1,1\} & N\\
L11n352\{0,0\} & N\\
L11n352\{1,1\} & N\\
L11n353\{0,0\} & N\\
L11n353\{1,0\} & N\\
L11n353\{0,1\} & N\\
L11n353\{1,1\} & N\\
L11n354\{0,0\} & N\\
L11n354\{1,0\} & N\\
L11n354\{0,1\} & N\\
L11n354\{1,1\} & N\\
L11n355\{0,0\} & N\\
\end{tabular}
\end{table}
\begin{table}[p]
\centering
\small
\setlength{\tabcolsep}{4pt}
\begin{tabular}[t]{cc}
name & strongly QP (SL)\\
\hline\hline
L11n355\{1,0\} & N\\
L11n355\{0,1\} & N\\
L11n355\{1,1\} & N\\
L11n356\{0,0\} & N\\
L11n356\{1,0\} & N\\
L11n356\{0,1\} & N\\
L11n356\{1,1\} & N\\
L11n357\{0,0\} & N\\
L11n357\{1,1\} & N\\
L11n358\{1,0\} & N\\
L11n358\{0,1\} & N\\
L11n359\{0,0\} & N\\
L11n359\{1,0\} & N\\
L11n359\{0,1\} & N\\
L11n359\{1,1\} & N\\
L11n360\{1,1\} & N\\
L11n361\{1,0\} & N\\
L11n361\{0,1\} & N\\
L11n363\{0,1\} & N\\
L11n364\{0,0\} & N\\
L11n364\{1,1\} & N\\
L11n365\{0,0\} & N\\
L11n365\{1,1\} & N\\
L11n366\{0,0\} & N\\
L11n366\{1,0\} & N\\
L11n366\{0,1\} & N\\
L11n366\{1,1\} & N\\
L11n367\{0,0\} & N\\
L11n367\{1,0\} & N\\
L11n367\{0,1\} & N\\
L11n367\{1,1\} & N\\
L11n368\{0,0\} & N\\
L11n368\{1,0\} & N\\
L11n368\{0,1\} & N\\
L11n368\{1,1\} & N\\
L11n369\{0,0\} & N\\
L11n369\{1,0\} & N\\
L11n369\{0,1\} & N\\
L11n369\{1,1\} & N\\
L11n370\{0,0\} & N\\
L11n370\{1,0\} & N\\
L11n370\{0,1\} & N\\
L11n370\{1,1\} & N\\
L11n371\{0,0\} & N\\
L11n372\{0,1\} & N\\
L11n373\{0,0\} & N\\
L11n375\{1,0\} & N\\
L11n375\{0,1\} & N\\
L11n376\{0,0\} & N\\
L11n376\{1,0\} & N\\
L11n376\{0,1\} & N\\
L11n376\{1,1\} & N\\
L11n377\{1,0\} & N\\
L11n377\{0,1\} & N\\
\end{tabular}%\hspace{.01\textwidth}
\begin{tabular}[t]{|cc}
name & strongly QP (SL)\\
\hline\hline
L11n378\{0,0\} & N\\
L11n378\{1,0\} & N\\
L11n378\{0,1\} & N\\
L11n378\{1,1\} & N\\
L11n379\{0,0\} & N\\
L11n379\{1,0\} & N\\
L11n379\{0,1\} & $\mathrm{Y}^{*}$ (7)\\
L11n379\{1,1\} & N\\
L11n380\{1,0\} & N\\
L11n380\{0,1\} & N\\
L11n382\{1,0\} & N\\
L11n382\{0,1\} & N\\
L11n383\{0,0\} & N\\
L11n383\{1,0\} & N\\
L11n383\{0,1\} & N\\
L11n383\{1,1\} & N\\
L11n384\{0,0\} & N\\
L11n384\{1,0\} & N\\
L11n384\{0,1\} & N\\
L11n384\{1,1\} & N\\
L11n385\{0,0\} & N\\
L11n385\{1,0\} & N\\
L11n385\{0,1\} & N\\
L11n385\{1,1\} & N\\
L11n386\{0,0\} & N\\
L11n386\{1,0\} & N\\
L11n386\{0,1\} & N\\
L11n386\{1,1\} & N\\
L11n387\{0,0\} & N\\
L11n387\{1,0\} & N\\
L11n387\{0,1\} & N\\
L11n387\{1,1\} & N\\
L11n388\{0,0\} & N\\
L11n388\{1,0\} & N\\
L11n388\{0,1\} & N\\
L11n388\{1,1\} & N\\
L11n389\{1,0\} & N\\
L11n389\{0,1\} & N\\
L11n390\{0,0\} & N\\
L11n390\{1,0\} & N\\
L11n390\{0,1\} & $\mathrm{Y}^{*}$ (7)\\
L11n391\{0,0\} & N\\
L11n391\{1,1\} & N\\
L11n392\{0,0\} & N\\
L11n392\{1,1\} & N\\
L11n393\{0,0\} & N\\
L11n393\{1,0\} & N\\
L11n393\{0,1\} & N\\
L11n393\{1,1\} & N\\
L11n394\{0,0\} & N\\
L11n394\{1,0\} & N\\
L11n394\{0,1\} & N\\
L11n394\{1,1\} & N\\
L11n395\{0,0\} & N\\
\end{tabular}%\hspace{.01\textwidth}
\begin{tabular}[t]{|cc}
name & strongly QP (SL)\\
\hline\hline
L11n395\{1,0\} & N\\
L11n395\{0,1\} & N\\
L11n395\{1,1\} & N\\
L11n397\{0,0\} & N\\
L11n397\{1,0\} & N\\
L11n397\{0,1\} & N\\
L11n397\{1,1\} & N\\
L11n398\{0,0\} & $\mathrm{Y}^{*}$ (5)\\
L11n398\{1,0\} & N\\
L11n398\{0,1\} & N\\
L11n399\{0,0\} & N\\
L11n399\{0,1\} & N\\
L11n400\{0,0\} & N\\
L11n400\{1,0\} & Y (7)\\
L11n400\{0,1\} & N\\
L11n401\{0,0\} & N\\
L11n401\{0,1\} & N\\
L11n402\{0,0\} & N\\
L11n402\{1,0\} & N\\
L11n403\{0,0\} & N\\
L11n403\{1,0\} & N\\
L11n403\{0,1\} & N\\
L11n403\{1,1\} & N\\
L11n405\{0,0\} & N\\
L11n405\{1,1\} & N\\
L11n408\{0,0\} & N\\
L11n408\{1,0\} & N\\
L11n409\{0,0\} & N\\
L11n409\{1,0\} & N\\
L11n409\{0,1\} & N\\
L11n409\{1,1\} & N\\
L11n410\{1,0\} & N\\
L11n410\{0,1\} & N\\
L11n411\{1,0\} & N\\
L11n411\{0,1\} & N\\
L11n411\{1,1\} & N\\
L11n412\{1,0\} & N\\
L11n413\{1,0\} & Y (7)\\
L11n414\{0,0\} & N\\
L11n416\{1,0\} & N\\
L11n417\{1,0\} & N\\
L11n417\{1,1\} & N\\
L11n418\{0,0\} & N\\
L11n418\{1,1\} & N\\
L11n419\{1,0\} & N\\
L11n419\{0,1\} & N\\
L11n419\{1,1\} & N\\
L11n420\{0,0\} & N\\
L11n420\{1,0\} & N\\
L11n420\{0,1\} & N\\
L11n420\{1,1\} & N\\
L11n421\{0,0\} & N\\
L11n421\{1,0\} & N\\
L11n421\{0,1\} & N\\
\end{tabular}
\end{table}
\begin{table}[p]
\centering
\small
\setlength{\tabcolsep}{4pt}
\begin{tabular}[t]{cc}
name & strongly QP (SL)\\
\hline\hline
L11n421\{1,1\} & N\\
L11n422\{0,0\} & N\\
L11n422\{1,1\} & N\\
L11n423\{0,0\} & N\\
L11n423\{0,1\} & N\\
L11n423\{1,1\} & N\\
L11n424\{0,0\} & N\\
L11n424\{1,0\} & N\\
L11n424\{0,1\} & N\\
L11n425\{0,0\} & N\\
L11n425\{0,1\} & N\\
L11n426\{1,0\} & N\\
L11n426\{1,1\} & N\\
L11n427\{1,0\} & N\\
L11n427\{0,1\} & N\\
L11n428\{0,0\} & N\\
L11n428\{0,1\} & N\\
L11n429\{0,0\} & N\\
L11n429\{1,0\} & N\\
L11n429\{0,1\} & N\\
L11n430\{0,0\} & N\\
L11n430\{1,0\} & N\\
L11n430\{0,1\} & N\\
L11n431\{0,0\} & N\\
L11n431\{1,0\} & N\\
L11n431\{0,1\} & N\\
L11n431\{1,1\} & N\\
L11n432\{0,0\} & N\\
L11n432\{1,0\} & N\\
L11n432\{1,1\} & N\\
L11n433\{0,0\} & N\\
L11n433\{1,0\} & N\\
L11n433\{1,1\} & N\\
L11n434\{0,0\} & N\\
L11n434\{1,0\} & N\\
L11n434\{1,1\} & N\\
L11n435\{0,0\} & N\\
L11n435\{1,0\} & N\\
L11n436\{0,0\} & N\\
L11n436\{1,0\} & N\\
L11n436\{0,1\} & N\\
L11n436\{1,1\} & N\\
L11n437\{0,0\} & N\\
L11n437\{0,1\} & N\\
L11n437\{1,1\} & N\\
L11n438\{0,1,0\} & N\\
L11n438\{1,1,0\} & N\\
L11n438\{0,0,1\} & N\\
L11n438\{0,1,1\} & N\\
L11n439\{0,0,0\} & N\\
L11n439\{0,1,0\} & N\\
\end{tabular}%\hspace{.01\textwidth}
\begin{tabular}[t]{|cc}
name & strongly QP (SL)\\
\hline\hline
L11n439\{1,1,0\} & N\\
L11n439\{0,0,1\} & N\\
L11n439\{0,1,1\} & N\\
L11n439\{1,1,1\} & N\\
L11n440\{0,0,0\} & N\\
L11n440\{1,0,0\} & N\\
L11n440\{1,1,0\} & N\\
L11n440\{1,0,1\} & Y (6)\\
L11n440\{0,1,1\} & N\\
L11n441\{1,0,0\} & N\\
L11n441\{0,1,0\} & N\\
L11n441\{0,0,1\} & N\\
L11n441\{1,0,1\} & N\\
L11n441\{0,1,1\} & N\\
L11n441\{1,1,1\} & N\\
L11n442\{0,1,0\} & N\\
L11n442\{1,1,0\} & N\\
L11n442\{0,0,1\} & N\\
L11n442\{0,1,1\} & N\\
L11n443\{1,0,1\} & N\\
L11n443\{0,1,1\} & N\\
L11n444\{1,0,0\} & N\\
L11n444\{1,1,0\} & N\\
L11n444\{0,1,1\} & N\\
L11n444\{1,1,1\} & N\\
L11n445\{0,1,0\} & N\\
L11n445\{0,1,1\} & N\\
L11n446\{0,1,0\} & N\\
L11n446\{0,1,1\} & N\\
L11n447\{1,0,0\} & N\\
L11n447\{0,1,0\} & N\\
L11n447\{1,1,0\} & N\\
L11n447\{1,0,1\} & N\\
L11n447\{0,1,1\} & N\\
L11n447\{1,1,1\} & N\\
L11n448\{0,0,0\} & N\\
L11n448\{1,0,0\} & N\\
L11n448\{0,1,0\} & N\\
L11n448\{1,1,0\} & N\\
L11n448\{0,0,1\} & N\\
L11n448\{1,0,1\} & N\\
L11n448\{0,1,1\} & N\\
L11n448\{1,1,1\} & N\\
L11n449\{0,0,0\} & N\\
L11n449\{1,0,0\} & N\\
L11n449\{0,1,0\} & N\\
L11n449\{0,0,1\} & N\\
L11n449\{1,0,1\} & N\\
L11n449\{0,1,1\} & N\\
L11n450\{0,0,0\} & N\\
L11n450\{0,1,0\} & N\\
\end{tabular}%\hspace{.01\textwidth}
\begin{tabular}[t]{|cc}
name & strongly QP (SL)\\
\hline\hline
L11n450\{1,1,0\} & N\\
L11n450\{0,0,1\} & N\\
L11n450\{0,1,1\} & N\\
L11n450\{1,1,1\} & N\\
L11n451\{1,0,0\} & N\\
L11n451\{1,1,0\} & N\\
L11n451\{0,0,1\} & N\\
L11n451\{1,0,1\} & N\\
L11n451\{0,1,1\} & N\\
L11n451\{1,1,1\} & N\\
L11n452\{1,0,0\} & N\\
L11n452\{1,1,0\} & N\\
L11n452\{0,0,1\} & N\\
L11n452\{1,0,1\} & N\\
L11n452\{0,1,1\} & N\\
L11n452\{1,1,1\} & N\\
L11n453\{1,0,0\} & N\\
L11n453\{0,1,0\} & N\\
L11n453\{1,1,0\} & N\\
L11n453\{0,0,1\} & N\\
L11n453\{1,0,1\} & N\\
L11n453\{0,1,1\} & N\\
L11n454\{1,0,0\} & N\\
L11n454\{0,1,0\} & N\\
L11n454\{1,1,0\} & N\\
L11n454\{0,0,1\} & N\\
L11n454\{1,0,1\} & N\\
L11n454\{0,1,1\} & N\\
L11n455\{1,0,0\} & N\\
L11n455\{0,1,0\} & N\\
L11n455\{1,1,0\} & N\\
L11n455\{0,0,1\} & N\\
L11n455\{1,0,1\} & N\\
L11n455\{0,1,1\} & N\\
L11n456\{1,0,0\} & N\\
L11n456\{0,1,0\} & N\\
L11n456\{1,1,0\} & N\\
L11n456\{0,0,1\} & $\mathrm{Y}^{*}$ (6)\\
L11n456\{1,0,1\} & N\\
L11n456\{0,1,1\} & N\\
L11n456\{1,1,1\} & N\\
L11n457\{1,1,0\} & N\\
L11n457\{0,0,1\} & N\\
L11n457\{1,1,1\} & N\\
L11n458\{0,1,0\} & N\\
L11n458\{1,1,0\} & N\\
L11n458\{0,0,1\} & N\\
L11n458\{0,1,1\} & N\\
L11n459\{0,0,1\} & N\\
 & \\
 & \\
\end{tabular}
\caption{All fibered prime links in LinkInfo with 11 or fewer crossings. $\text{Y}$ indicates that the given LinkInfo link is strongly quasi-positive, $\text{Y}^{*}$ that its mirror is, and $\text{N}$ that neither is. Parentheses give the maximal self-linking number of the strongly quasi-positive link.}
\label{tab:small_links}
\end{table}
\clearpage
\endgroup

\end{document}